\documentclass{article}
\pdfoutput=1
\usepackage{euscript,amsmath,amssymb,amsfonts,graphicx,bm,amsthm}
\usepackage{cite}

  \usepackage{paralist}
  \usepackage{graphics} 
\usepackage[caption=false]{subfig}
\usepackage{soul}

\usepackage{latexsym,amssymb,enumerate,amsmath,amsthm} 
  \usepackage{graphicx} 
  \usepackage{tikz}
     \usepackage[colorlinks=false]{hyperref}
 \hypersetup{urlcolor=blue, citecolor=red}
\usepackage{latexsym,amssymb,enumerate,amsmath} 
\usepackage{pgfplots}
\usepackage{soul,xcolor}

\usepackage{mathtools}

\newcommand{\dd}{\textup{d}}
\def\eps{\varepsilon}
\def\E{\mathbb{E}}
\def\P{\mathbb{P}}
\def\R{\mathbb{R}}

\def\L{\mathcal{L}}
\newcommand{\deuc}{L_{\textup{euc}}}
\newcommand{\drie}{L_{\textup{rie}}}

\newtheorem{theorem}{Theorem}
\newtheorem{proposition}[theorem]{Proposition}

\theoremstyle{plain}
\theoremstyle{remark}

\theoremstyle{definition}
\newtheorem*{definition*}{Definition}

\begin{document}


\title{Extreme hitting probabilities for diffusion}


\author{Samantha Linn\thanks{Department of Mathematics, University of Utah, Salt Lake City, UT 84112 USA.} \and Sean D. Lawley\thanks{Department of Mathematics, University of Utah, Salt Lake City, UT 84112 USA (\texttt{lawley@math.utah.edu}). SDL was supported by the National Science Foundation (Grant Nos.\ DMS-1944574 and DMS-1814832).}
}
\date{\today}
\maketitle

\begin{abstract}
A variety of systems in physics, chemistry, biology, and psychology are modeled in terms of diffusing ``searchers'' looking for ``targets.'' Examples range from gene regulation, to cell sensing, to human decision-making. A commonly studied statistic in these models is the so-called hitting probability for each target, which is the probability that a given single searcher finds that particular target. However, the decisive event in many systems is not the arrival of a given single searcher to a target, but rather the arrival of the fastest searcher to a target out of many searchers. In this paper, we study the probability that the fastest diffusive searcher hits a given target in the many searcher limit, which we call the extreme hitting probability. We first prove an upper bound for the decay of the probability that the searcher finds a target other than the closest target. This upper bound applies in very general settings and depends only on the relative distances to the targets. Furthermore, we find the exact asymptotics of the extreme hitting probabilities in terms of the short-time distribution of when a single searcher hits a target. These results show that the fastest searcher always hits the closest target in the many searcher limit. While this fact is intuitive in light of recent results on the time it takes the fastest searcher to find a target, our results give rigorous, quantitative estimates for the extreme hitting probabilities. We illustrate our results in several examples and numerical simulations.
\end{abstract}

\section{Introduction}

Many systems in physics, chemistry, biology, and psychology have been modeled in terms of diffusing ``searchers'' finding ``targets'' \cite{redner2001}. Examples include diffusion-limited chemical reactions \cite{hanggi1990}, immune response initiation from a T cell finding an antigen-presenting cell in a lymph node \cite{Delgado2015}, gene activation from a transcription factor finding the corresponding gene \cite{Larson2011}, and making a decision when the amount of evidence in favor of a certain response surpasses a given threshold \cite{ratcliff2016}.

To understand the timescales in these systems, one often studies the time it takes a searcher to find a target, which is called the first passage time (FPT). If there are multiple targets, then another important quantity is the probability that a searcher finds a particular target, which is called the hitting probability or splitting probability \cite{condamin2007}. 
For example, cells sense their environment through diffusive signals (searchers) arriving at membrane receptors (targets), and the receptor hitting probabilities have been used to study how cells could infer the location of the source of the signal \cite{lawley2020prl} (intuitively, if most of the diffusive signal hits receptor $k$, then the source is likely near that receptor). As another example, decision-making has long been modeled in the psychology literature in terms of a diffusive searcher moving between targets which represent choices for the decision \cite{ratcliff1978}. In these models, the hitting probabilities thus describe the likelihood that a particular decision will be made \cite{ratcliff2008}. 

To describe these scenarios more precisely, let $\{X(t)\}_{t\ge0}$ denote the path of a diffusive searcher among $m\ge2$ targets denoted by $V_{0},V_{1},\dots,V_{m-1}$. The FPT of the searcher to one of the targets is then
\begin{align}\label{tau0}
\tau
:=\inf\{t>0:X(t)\in\cup_{k=0}^{m-1}V_{k}\}.
\end{align}
If $\kappa\in\{0,1,\dots,m-1\}$ denotes the index of the target hit by the searcher (i.e.\ $\kappa=k$ if $X(\tau)\in V_{k}$), then the hitting probabilities for the $m$ targets are the values of
\begin{align}\label{hit0}
\P(\kappa=k)\quad \text{for }k\in\{0,1,\dots,m-1\}.
\end{align}

Mathematically, finding the hitting probabilities in \eqref{hit0} for a single diffusive searcher requires solving an elliptic partial differential equation (PDE) with mixed boundary conditions. To illustrate, consider a purely diffusive searcher in a bounded domain $M\subset\R^{d}$ with a reflecting boundary containing $m\ge2$ targets $V_{0},V_{1},\dots,V_{m-1}\subset M$ (see Figure~\ref{figschem0} for an illustration). Conditioned that the searcher starts at $x_{0}\in M$, the probability that the searcher hits target $k$ first,
\begin{align*}
\pi(x_{0})
:=\P(\kappa=k\,|\,X(0)=x_{0}),
\end{align*}
satisfies the PDE boundary value problem \cite{oksendal2003},
\begin{align}\label{pde}
\begin{split}
\Delta \pi
&=0,\quad x_{0}\in M\backslash \cup_{j=0}^{m-1}V_{j},\\
\pi
&=1,\quad x_{0}\in V_{k},\\
\pi
&=0,\quad x_{0}\in\cup_{j\neq k}V_{j},
\end{split}
\end{align}
with reflecting boundary conditions on the boundary of $M$ (if the searcher experiences drift or a space-dependent diffusion coefficient, then the Laplacian $\Delta$ in \eqref{pde} is replaced by a more complicated differential operator \cite{oksendal2003}). Hence, finding the hitting probabilities in \eqref{hit0} generally amounts to solving a PDE akin to \eqref{pde}. Finding moments of the FPT $\tau$ in \eqref{tau0} involves solving similar PDEs to \eqref{pde} \cite{gardiner2009}. Analyzing such FPT and hitting probability problems has generated a great deal of PDE analysis \cite{cheviakov11, ward10, ward10b, chen2011, holcman2014, grebenkov2016, lindsay2017, lawley2019dtmfpt}.

The majority of prior studies of diffusive search have considered a single searcher. However, it has recently been emphasized that the important timescale in many systems is not the FPT of a single searcher, but rather the fastest FPT out of many searchers \cite{schuss2019, lawley2020uni}. That is, if there are $N\gg1$ searchers with respective FPTs $\tau_{1},\dots,\tau_{N}$, then the decisive timescale is the so-called fastest FPT or extreme FPT,
\begin{align}\label{TN0}
T_{N}
:=\min\{\tau_{1},\dots,\tau_{N}\}.
\end{align}
As two examples, human fertilization is triggered when the fastest sperm cell out of $N\approx10^{8}$ sperm cells finds the egg \cite{meerson2015}, and a gene regulatory response is determined by only the fastest few transcription factors out of $N\in[10^{2},10^{4}]$ transcription factors relaying the signal \cite{harbison2004}. For more examples, see the review \cite{schuss2019} and the subsequent commentaries \cite{coombs2019, tamm2019, martyushev2019, rusakov2019, sokolov2019, redner2019, basnayake2019c}.

\begin{figure}
  \centering
             \includegraphics[width=0.5\textwidth]{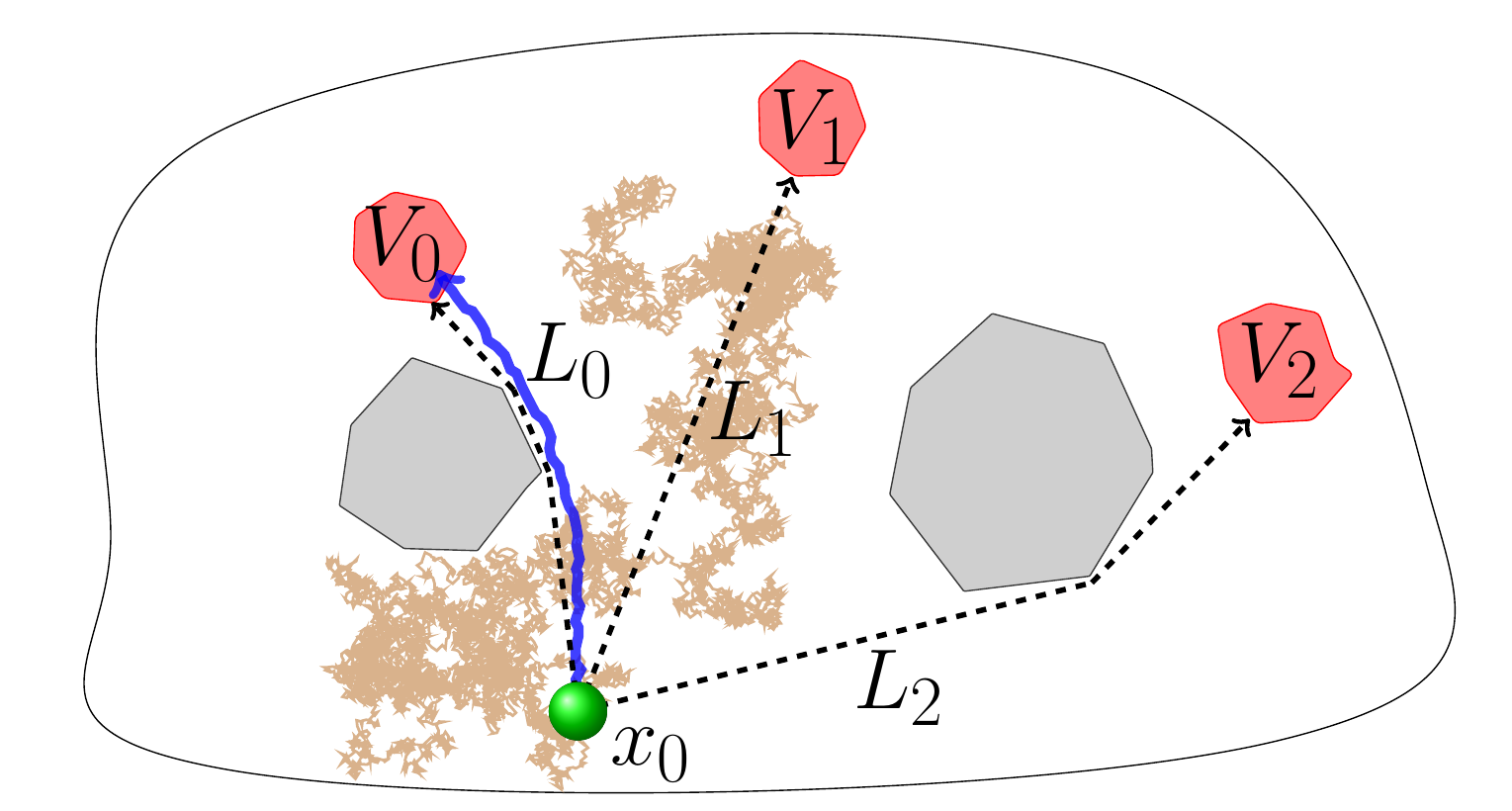}
 \caption{Schematic diagram of diffusive search. The searchers start at the green ball labeled $x_{0}$ and diffuse until they hit one of the $m=3$ targets (red regions) labeled $V_{0}$, $V_{1}$, and $V_{2}$. The dashed lines with lengths $L_{0}<L_{1}<L_{2}$ show the shortest paths to each of the targets which avoid the reflecting obstacles (gray regions). The brown trajectory depicts the path of a typical searcher which wanders around the domain before finding a target. The blue trajectory depicts the path of the fastest searcher out of $N\gg1$ searchers which tends to follow the shortest path to the closest target.}
 \label{figschem0}
\end{figure}

In this case of $N\gg1$ searchers, a statistic related to the extreme FPT in \eqref{TN0} is what we call the extreme hitting probability. More precisely, let $\kappa_{n}\in\{0,\dots,m-1\}$ indicate the target hit by the $n$th searcher. If $n^{*}\in\{1,\dots,N\}$ denotes the index of the fastest searcher (meaning $\tau_{n^{*}}=T_{N}$), then let
\begin{align*}
K_{N}
=\kappa_{n^{*}}\in\{0,\dots,m-1\}
\end{align*}
indicate the target hit by this fastest searcher. The extreme hitting probabilities are then
\begin{align*}
\P(K_{N}=k)\quad\text{for }k\in\{0,1,\dots,m-1\}.
\end{align*}
In the cell sensing model described above \cite{lawley2020prl}, the extreme hitting probabilities describe the distribution of where the first signaling molecules are likely to hit the cell. In decision-making models, the extreme hitting probabilities describe choices made by early adopters, which can affect the subsequent decisions made by a larger population \cite{mann2018, mann2020, karamched2020}.

In this paper, we study the extreme hitting probabilities for $N\gg1$ independent and identically distributed (iid) diffusive searchers.  If $0$ denotes the index of the target closest to the searcher starting location(s), then we prove that the probability that the fastest searcher finds target $k\neq0$ vanishes according to
\begin{align}\label{vanishlog}
\P(K_{N}=k)
=o(N^{1-(L_{k}/L_{0})^{2}+\eps})\quad\text{as }N\to\infty\text{ for any $\eps>0$},
\end{align}
where $f=o(g)$ denotes $f/g\to0$. In \eqref{vanishlog}, $L_{j}>0$ denotes a certain geodesic distance between the searcher starting locations and target $j\in\{0,\dots,m-1\}$, and we assume $L_{0}<L_{k}$. Roughly speaking, $L_{j}$ is the shortest distance the searcher must travel to hit target $j$, as illustrated in Figure~\ref{figschem0} (the geodesic distance is given precisely in section~\ref{general}). We prove that \eqref{vanishlog} holds in quite general settings, including for $d$-dimensional diffusion processes (i) with space-dependent diffusion coefficients and drifts, (ii) on Riemannian manifolds, (iii)
with reflecting obstacles, and (iv) with partially absorbing targets.

Moreover, the result in \eqref{vanishlog} can be sharpened under additional assumptions on the short-time behavior of the joint probability distribution of $(\tau,\kappa)$ for a single searcher. In particular, we prove that
\begin{align}\label{vanish0}
\P(K_{N}=k)
\sim{{\eta}} (\ln N)^{\rho}N^{1-(L_{k}/L_{0})^{2}}\quad\text{as }N\to\infty\quad\text{for }k\neq0,
\end{align}
where the constant ${{\eta}}>0$ and the logarithmic power $\rho\in\R$ are given explicitly in terms of parameters in the short-time distribution of $(\tau,\kappa)$. Throughout this paper, $f\sim g$ denotes $f/g\to1$.

The results in \eqref{vanishlog}-\eqref{vanish0} show that the fastest searcher always hits the closest target in the limit of many searchers. While this fact is intuitive in light of recent results on extreme FPTs \cite{lawley2020uni}, the bound in \eqref{vanishlog} and the exact asymptotics in \eqref{vanish0} give rigorous, quantitative estimates for the extreme hitting probabilities. We now highlight two features of these estimates.

First, \eqref{vanishlog} is a general result that requires knowing merely the distance to the closest target, $L_{0}$, and the distance to the $k$th target, $L_{k}$ (see Figure~\ref{figschem0}). In particular, one can use \eqref{vanishlog} to estimate the extreme hitting probabilities without detailed knowledge of the geometry, diffusivity, drift, etc. This is in stark contrast to obtaining the hitting probabilities for a single searcher, which requires solving an elliptic PDE with mixed boundary conditions as in \eqref{pde}.

Second, \eqref{vanishlog}-\eqref{vanish0} show how relatively small differences in target distances yield vastly different extreme hitting probabilities. For example, consider $N$ iid searchers which move by pure diffusion in one dimension between a target at $x=0$ and a target at $x=l>0$ (see the left panel of Figure~\ref{fig1dintro}). If the searchers start in the left half of the interval, $x_{0}\in(0,l/2)$, then the respective distances to each target are simply
\begin{align*}
L_{0}
=x_{0}
<L_{1}
=l-x_{0}.
\end{align*}
In the case of a single searcher ($N=1$), it is well-known that the probability that the searcher hits $x=l$ before $x=0$ is a linear function of the starting location \cite{oksendal2003},
\begin{align*}
\P(K_{1}=1)
=\P(\kappa=1)
=\frac{x_{0}}{l}
=\frac{1}{1+L_{1}/L_{0}}.
\end{align*}
This means that if a searcher starts only slightly closer to $x=0$ than $x=l$, then that single searcher is only slightly more likely to hit $x=0$ before $x=l$. However, if there are $N\gg1$ such searchers, then we apply \eqref{vanish0} to this example and find that
\begin{align}\label{est0}
&\P(K_{N}=1)
\sim
{{\eta}}
(\ln N)^{\rho}
N^{1-{\beta}}
\quad\text{as }N\to\infty,
\end{align}
where 
\begin{align*}
{\beta}
&=\Big(\frac{L_{1}}{L_{0}}\Big)^{2}
>1,\quad 
\rho=\frac{\beta-1}{2},\quad
{{\eta}}
=\sqrt{\pi^{\beta-1}\beta}\Gamma(\beta)
>0.
\end{align*}
In the right panel of Figure~\ref{fig1dintro}, we plot the estimate in \eqref{est0} against numerical simulations, which illustrates the rapid decay of $\P(K_{N}=1)$ even if $L_{0}$ is only slightly less than $L_{1}$. See section~\ref{ex1d} for details on this example.

\begin{figure}
  \centering
                 \includegraphics[width=0.465\textwidth]{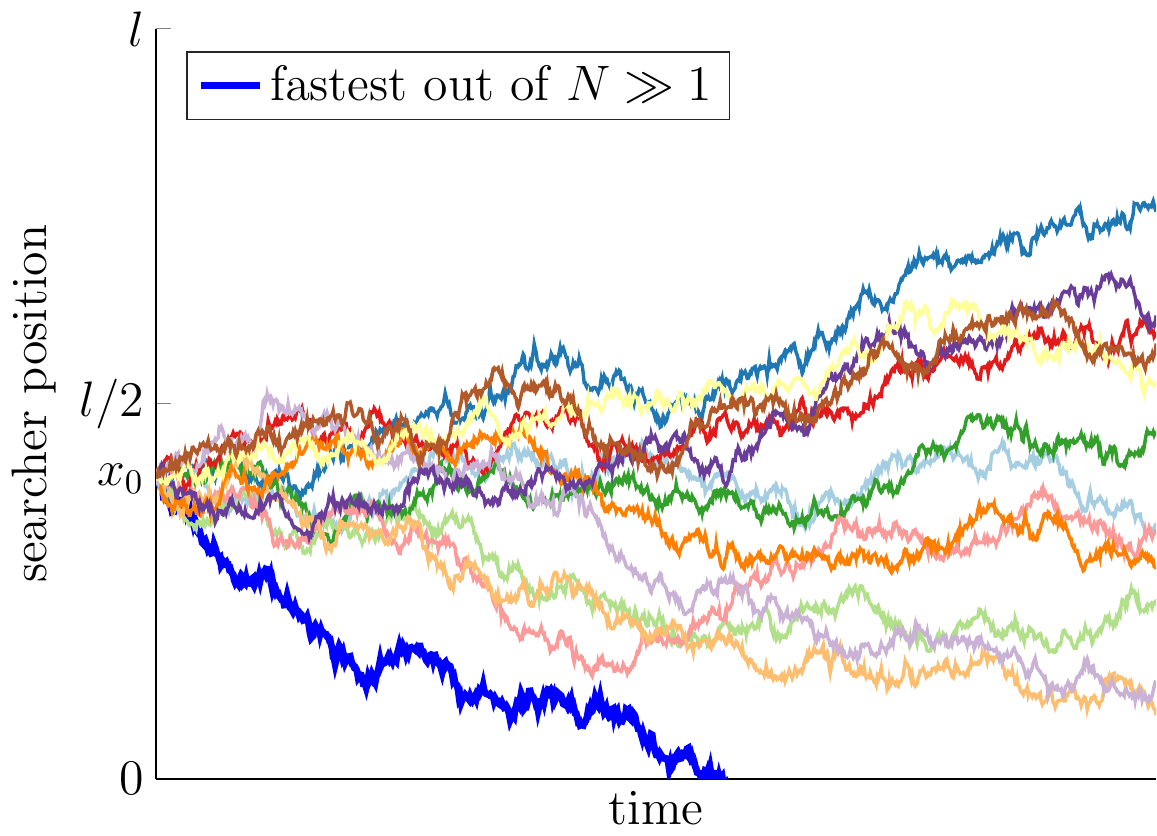}
                              \qquad
             \includegraphics[width=0.465\textwidth]{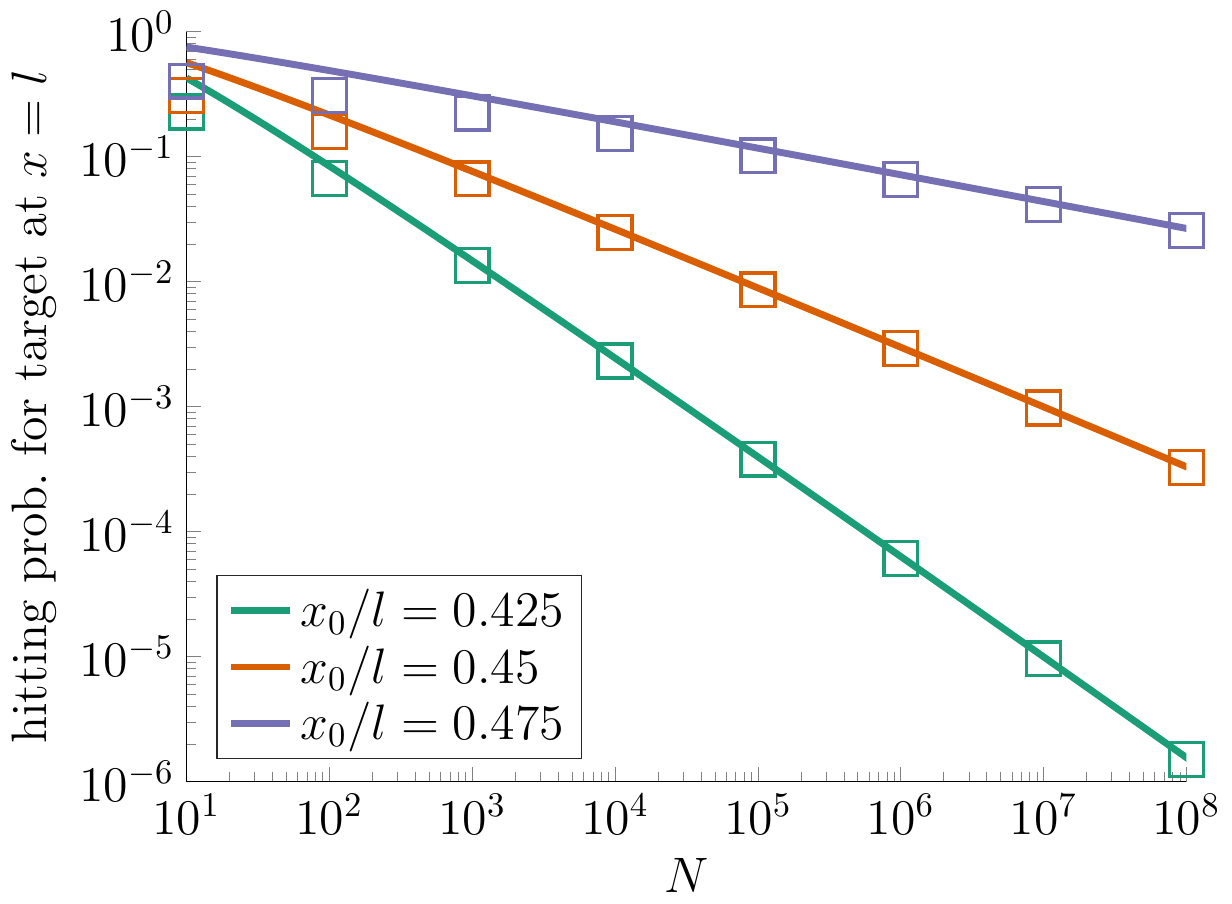}
 \caption{Diffusive search in the interval $(0,l)$. \textbf{Left:} The thin curves depict paths of typical searchers which wander around the interval and the thick blue curve depicts the path of the fastest searcher which quickly hits the closer target. \textbf{Right:} The curves plot the asymptotic estimate in \eqref{est0} for the probability that the fastest searcher out of $N$ searchers hits the target at $x=l$ for different values of the relative searcher starting location $x_{0}/l$. The square markers plot the value of this extreme hitting probability obtained from numerical simulations. See section~\ref{ex1d} for details.}
 \label{fig1dintro}
\end{figure}

The rest of the paper is organized as follows. In section~\ref{math}, we first  represent the extreme hitting probabilities as an integral involving the probability distribution of $\tau$ and the joint probability distribution of $(\tau,\kappa)$. We then find the large $N$ asymptotics of this integral under some assumptions on the short-time behavior of these probability distributions. In section~\ref{examples}, we illustrate the exact asymptotic estimate in \eqref{vanish0} in several examples and compare to numerical simulations. In section~\ref{general}, we prove that the bound in \eqref{vanishlog} holds in several very general settings. We conclude by discussing related work. We present the mathematical proofs along with some technical points in the appendix.

\section{Hitting probability asymptotics}\label{math}

In this section, we prove results on the asymptotics of hitting probabilities under general assumptions on the short-time behavior of hitting time distributions. The theorems in this section make no reference to diffusion. Rather, the theorems merely assume certain short-time behavior for hitting time distributions. We then show in sections~\ref{examples} and \ref{general} that this behavior is characteristic of diffusive search.

\subsection{Probabilistic setup and integral representation}\label{splitting}

Let $\tau>0$ be a nonnegative random variable and let $\kappa$ be a random variable taking values in the finite set $\{0,1,\dots,m-1,\infty\}$. In the applications of interest, $\tau$ is the FPT of a searcher to a target and $\kappa$ indicates which of the $m\ge2$ targets that the searcher finds. We set $\kappa=\infty$ if $\tau=\infty$, which describes the event that the searcher never finds a target (the event $\tau=\infty$ occurs with positive probability in, for example, diffusive search in an unbounded domain in dimension $d\ge3$).

For each target index $k\in\{0,\dots,m-1\}$, define
\begin{align*}
F_{k}(t)
:=\P(\tau\le t \cap \kappa=k),\quad k\in\{0,\dots,m-1\},\,t\in\R.
\end{align*}
Furthermore, let $F(t)$ denote the cumulative distribution function of $\tau$,
\begin{align*}
F(t)
:=\P(\tau\le t)
=\sum_{k=0}^{m-1}F_{k}(t),\quad t\in\R.
\end{align*}
We assume $F(t)$ is a continuous function with $F(0)=0$, which ensures that $\P(\tau=t)=0$ for every $t\in\R$.

Let $\{(\tau_{n},\kappa_{n})\}_{n\ge1}$ be an iid sequence of realizations of
\begin{align*}
(\tau,\kappa)\in(0,\infty]\times\{0,1,\dots,m-1,\infty\}.
\end{align*}
Define the fastest FPT for any $N\ge1$,
\begin{align*}
T_{N}
:=\min\{\tau_{1},\dots,\tau_{N}\}.
\end{align*}
Furthermore, let
\begin{align*}
K_{N}\in\{0,\dots,m-1,\infty\}
\end{align*}
denote the index of the target hit by the fastest searcher. That is, if $T_{N}=\tau_{n^{*}}$ for some $n^{*}\in\{1,\dots,N\}$, then
\begin{align}\label{KN}
K_{N}=\kappa_{n^{*}}.
\end{align}
We note that event $\tau_{n^{*}}=\tau_{n'}<\infty$ for $n^{*}\neq n'$ has probability zero since $F(t)$ is continuous. Further, if $T_{N}=\infty$, then $\kappa_{n}=\infty$ for all $n\in\{1,\dots,N\}$ and thus $K_{N}=\infty$. Hence, there is no ambiguity in \eqref{KN} and $K_{N}$ is well-defined. We emphasize that $K_{N}$ is the index of the target hit by the fastest searcher, whereas $\kappa_{N}$ is the index of the target hit by the $N$th searcher.

We are interested in the distribution of $K_{N}$ for large $N$. The following proposition represents the distribution of $K_{N}$ in a form which is convenient for analyzing the large $N$ limit.

\begin{proposition}\label{pint}
Under the assumptions of section~\ref{splitting}, the distribution of $K_{N}$ can be written as the following Riemann--Stieltjes integral, 
\begin{align*}
\P(K_{N}=k)
=N\int_{0}^{\infty}\big(1-F(t)\big)^{N-1}\,\dd F_{k}(t),\quad k\in\{0,1,\dots,m-1\}.
\end{align*}
Further, $\P(K_{N}=\infty)=(\P(\tau=\infty))^{N}$, where $\P(\tau=\infty)=1-\lim_{t\to\infty}F(t)$.
\end{proposition}

The proof of Proposition~\ref{pint}, as well as the proofs of all the results in this section, are given in the appendix.

\subsection{Extreme hitting probabilities}

Since $F(t)$ is a nondecreasing function, it is clear from the form of the integral in Proposition~\ref{pint} that the large $N$ asymptotics of $\P(K_{N}=k)$ depend chiefly on the behavior of $F(t)$ and $F_{k}(t)$ near $t=0$. The following proposition computes the exact asymptotics of an integral whose integrand is typical of the integrand in Proposition~\ref{pint} near $t=0$ for the case of diffusive search.

We remind the reader that $f\sim g$ denotes $f/g\to1$.

\begin{proposition}\label{p1}
Assume $C_{+}>C>0$, $A>0$, and $p,q\in\R$. Then there exists a $\delta_{0}>0$ so that for all $\delta\in(0,\delta_{0}]$, we have
\begin{align*}
&\int_{0}^{\delta}t^{q-2}e^{-C_{+}/t}\big(1-At^{p}e^{-C/t}\big)^{N-1}\,\dd t
\sim
{{\eta}}
(\ln N)^{p{\beta}-q}
N^{-{\beta}}
\quad\text{as }N\to\infty,
\end{align*}
where
\begin{align*}
{\beta}
&=C_{+}/C>1,\quad
{{\eta}}
=C^{q-1}(AC^{p})^{-\beta}\Gamma(\beta)>0,
\end{align*}
and $\Gamma({\beta}):=\int_{0}^{\infty}z^{{\beta}-1}e^{-z}\,\dd z$ denotes the gamma function.
\end{proposition}

The following theorem uses Proposition~\ref{p1} to compute the exact asymptotics of the distribution of $K_{N}$ assuming fairly detailed knowledge of the short-time behavior of $F$ and $F_{k}$. In particular, the theorem assumes the short-time asymptotics of $F$ and $F_{k}$ are known on a linear scale. We show in section~\ref{examples} that this short-time behavior of $F$ and $F_{k}$ is typical of diffusive search.

\begin{theorem}\label{detailprob}
Under the assumptions of section~\ref{splitting}, assume further that for some $k\in\{1,\dots,m-1\}$,
\begin{align}
F(t)
&\sim At^{p}e^{-C_{0}/t}\quad\text{as }t\to0+,\label{s1}\\
F_{k}(t)
&\sim Bt^{q}e^{-C_{k}/t}\quad\text{as }t\to0+,\label{s2}
\end{align}
where $C_{k}>C_{0}>0$, $A>0$, $B>0$, and $p,q\in\R$. Then
\begin{align*}
&\P(K_{N}=k)
\sim
{{\eta}}
(\ln N)^{p{\beta}-q}
N^{1-{\beta}}
\quad\text{as }N\to\infty,
\end{align*}
where
\begin{align}\label{bk}
{\beta}
&:=C_{k}/C_{0}>1,\quad
{{\eta}}
:=B(C_{0})^{q-p\beta}A^{-\beta}\beta\Gamma(\beta)
>0,
\end{align}
and $\Gamma({\beta}):=\int_{0}^{\infty}z^{{\beta}-1}e^{-z}\,\dd z$ denotes the gamma function.
\end{theorem}

We show in sections~\ref{examples} and \ref{general} that the constants $C_{0}$ and $C_{k}$ in Theorem~\ref{detailprob} are related to the shortest distances to the closest target (taken to be target $0$) and the $k$th target. In particular, we show that it is generally the case that
\begin{align*}
C_{j}
=\frac{(L_{j})^{2}}{4D}>0\quad j\in\{0,\dots,m-1\},
\end{align*}
where $L_{j}>0$ is the shortest distance from the searcher starting location(s) (in an appropriately chosen notion of distance) and $D>0$ is a characteristic diffusion coefficient. Hence, Theorem~\ref{detailprob} gives the large $N$ behavior of the probability that the fastest searcher finds a target other than the closest target. Of course, it follows that the probability that the fastest searcher finds the closest target converges to unity at a rate determined by the next closest target(s). We further note that Proposition~\ref{pint} ensures that the probability that the fastest searcher does not hit any target vanishes exponentially fast,
\begin{align}\label{vf}
\P(K_{N}=\infty)
=(\P(\tau=\infty))^{N}\to0\quad\text{as }N\to\infty,
\end{align}
apart from the trivial case that $\P(\tau=\infty)=1$. In particular, under the assumptions of Theorem~\ref{detailprob}, the decay in \eqref{vf} holds and thus
\begin{align*}
\P(K_{N}=\infty)
=o(\P(K_{N}=k))\quad\text{as $N\to\infty$ for any $k\in\{0,1,\dots,m-1\}$},
\end{align*}
since \eqref{s1} implies $\P(\tau=\infty)<1$. We remind the reader that $f=o(g)$ denotes $f/g\to0$.

In general scenarios, the short-time asymptotics of $F$ and $F_{k}$ required by Theorem~\ref{detailprob} may not be known. The following theorem gives an upper bound on the decay of the distribution of $K_{N}$ assuming one merely has bounds on the short-time behavior of $F$ and $F_{k}$ on a logarithmic scale. We show in section~\ref{general} that these bounds hold in very general settings for diffusive search.

\begin{theorem}\label{logprob}
Under the assumptions of section~\ref{splitting}, assume further that for some $k\in\{1,\dots,m-1\}$,
\begin{align}\label{lb}
\lim_{t\to0+}t\ln F(t)
&\ge -C_{0}<0,\quad
\lim_{t\to0+}t\ln F_{k}(t)
\le -C_{k}<0,
\end{align}
where $C_{k}>C_{0}>0$. Then for every $\eps>0$, 
\begin{align}\label{ub}
&\P(K_{N}=k)
=o(N^{1-{\beta}+\eps})
\quad\text{as }N\to\infty,
\end{align}
where
\begin{align*}
{\beta}
&:=C_{k}/C_{0}>1.
\end{align*}
If we assume further that
\begin{align}\label{equal}
\lim_{t\to0+}t\ln F(t)
&= -C_{0}<0,\quad
\lim_{t\to0+}t\ln F_{k}(t)
= -C_{k}<0,
\end{align}
then in addition to \eqref{ub}, we also have that for every $\eps>0$,
\begin{align}\label{both}
N^{1-{\beta}-\eps}
=o(\P(K_{N}=1))
\quad\text{as }N\to\infty.
\end{align}
\end{theorem}

We note that \eqref{lb} implies that $\P(\tau=\infty)<1$ and thus the decay in \eqref{vf} holds under the assumptions of Theorem~\ref{logprob}.

\section{Examples and numerical simulations}\label{examples}

Theorem~\ref{detailprob} yields the exact asymptotics of the extreme hitting probabilities as $N\to\infty$ in terms of the short-time behavior of $F(t)$ and $F_{k}(t)$. In this section, we illustrate these results and compare them to numerical simulations in several examples.

\subsection{Pure diffusion in one dimension}\label{ex1d}

Consider pure diffusion with diffusivity $D>0$ in one dimensional space $\R$. Suppose each searcher starts at $x_{0}\in(0,l)$ and the targets are at the left and right of the interval $(0,l)$ and are denoted by
\begin{align*}
V_{0}:=(-\infty,0],\quad
V_{1}:=[l,\infty).
\end{align*}
If there is only one searcher, then it is well-known that the probability that this single searcher reaches $V_{1}$ before $V_{0}$ is
\begin{align*}
\P(K_{1}=1)
=\frac{x_{0}}{l}
=1-\P(K_{1}=0).
\end{align*}

We now approximate the probability that the fastest searcher out of $N\gg1$ searchers finds $V_{1}$ before $V_{0}$. In the simple case that the searchers start exactly in the center of the interval, symmetry implies
\begin{align*}
\P(K_{N}=1)
=\P(K_{N}=0)
=1/2\quad\text{for all $N\ge1$ if $x_{0}=l/2$}.
\end{align*}
To understand the behavior of $K_{N}$ apart from this case, we need information about the short-time behavior of $F_{k}$. Without loss of generality, assume the searchers start in the left half of the interval,
\begin{align*}
x_{0}\in(0,l/2),
\end{align*}
and define the lengths from $x_{0}$ to the respective targets,
\begin{align}\label{ls}
0<L_{0}:=x_{0}<l-x_{0}=:L_{1}.
\end{align}
In this case, one can show that (see the Appendix)
\begin{align}\label{shortF}
F(t)
:=\P(\tau\le t)\sim At^{p}e^{-C_{0}/t}\quad\text{as }t\to0+,
\end{align}
where
\begin{align}\label{shortFA}
A
=\sqrt{\frac{4D}{\pi (L_{0})^{2}}},\quad 
p
=\frac{1}{2},\quad
C_{0}
=\frac{(L_{0})^{2}}{4D}.
\end{align}
One can also show that (see the Appendix)
\begin{align}\label{shortfar}
F_{1}(t)
\sim Bt^{q}e^{-C_{1}/t}
\quad\text{as }t\to0+,
\end{align}
where
\begin{align}\label{shortfarB}
B
=\sqrt{\frac{4D}{\pi (L_{1})^{2}}},\quad 
q
=p
=\frac{1}{2},\quad
C_{1}
=\frac{(L_{1})^{2}}{4D}.
\end{align}

Therefore, Theorem~\ref{detailprob} implies that \begin{align}\label{1das}
&\P(K_{N}=1)
\sim
{{\eta}}
(\ln N)^{(\beta-1)/2}
N^{1-{\beta}}
\quad\text{as }N\to\infty,
\end{align}
where 
\begin{align*}
{\beta}
&=\Big(\frac{L_{1}}{L_{0}}\Big)^{2}
=\Big(\frac{l-x_{0}}{x_{0}}\Big)^{2}
>1,\\
{{\eta}}
&=\sqrt{\beta\pi^{\beta-1}}\Gamma(\beta)
>0.
\end{align*}

In Figure~\ref{fig1dex}, we compare \eqref{1das} to numerical simulations. In the left panel, the solid curves are the asymptotic formula in \eqref{1das} and the square markers are computed using numerical integration of the representation for $\P(K_{N}=1)$ given in Proposition~\ref{pint} (see the Appendix for details of the numerical method). The right panel plots the relative error between the asymptotic formula in \eqref{1das} and the value of $\P(K_{N}=1)$ obtained from numerical integration,
\begin{align}\label{re}
\bigg|\frac{\P(K_{N}=1)-{{\eta}}
(\ln N)^{(\beta-1)/2}
N^{1-{\beta}}}{\P(K_{N}=1)}\bigg|.
\end{align}

\begin{figure}
  \centering
             \includegraphics[width=1\textwidth]{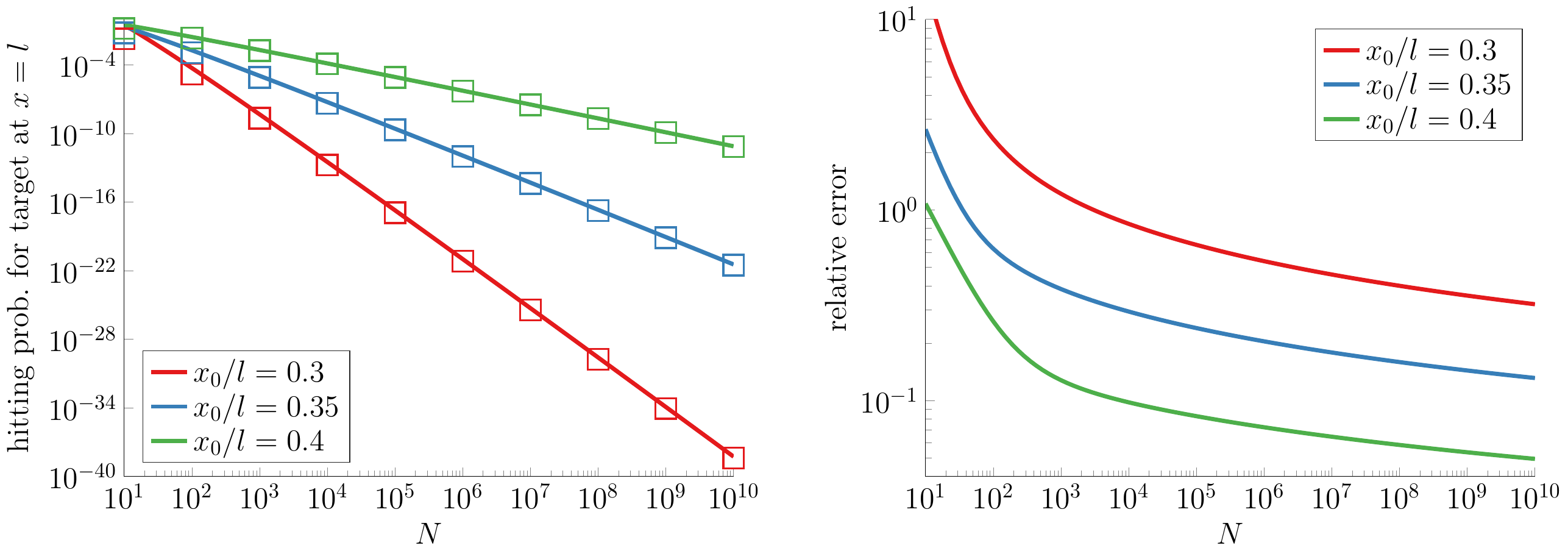}
 \caption{Extreme hitting probabilities for pure diffusion in the interval $(0,l)$. See section~\ref{ex1d} for details.}
 \label{fig1dex}
\end{figure}

Before moving to the next example, we briefly point out that \eqref{1das} yields the large $N$ asymptotics of the solution to Laplace's equation in the $N$-dimensional hypercube with certain mixed boundary conditions. In particular, let $\mathbf{x}=(x_{1},\dots,x_{N})\in\R^{N}$ denote an $N$-dimensional vector and suppose the function $u(\mathbf{x})$ is harmonic in $(0,l)^{N}$,
\begin{align*}
\Delta u
&=0,\quad \mathbf{x}\in(0,l)^{N}.
\end{align*}
Suppose further that $u$ satisfies the boundary conditions
\begin{align*}
u(\mathbf{x})
&=0\quad\text{if $x_{n}=0$ for some $n\in\{1,\dots,N\}$ and $x_{i}\in(0,l)$ for $i\neq n$},\\
u(\mathbf{x})
&=1\quad\text{if $x_{n}=l$ for some $n\in\{1,\dots,N\}$ and $x_{i}\in(0,l)$ for $i\neq n$}.
\end{align*}
Since $N$ independent diffusive searchers in the interval $(0,l)$ is equivalent to a single diffusive searcher in $(0,l)^{N}\in\R^{N}$, it follows that \cite{oksendal2003}
\begin{align*}
u((x_{0},x_{0},\dots,x_{0}))
=\P(K_{N}=1).
\end{align*}
Hence, \eqref{1das} yields the large $N$ behavior of $u((x_{0},x_{0},\dots,x_{0}))$. The analogous result for the asymptotics of solutions to similar high-dimensional elliptic PDEs holds for the examples given below.

\subsection{Diffusion with drift in one dimension}\label{drift}

Consider the example in section~\ref{ex1d}, but now suppose that each searcher experiences a constant drift $\mu\in\R$. Precisely, suppose the position $\{X(t)\}_{t\ge0}$ of a searcher evolves according to the stochastic differential equation,
\begin{align*}
\dd X(t)
=\mu\,\dd t
+\sqrt{2D}\,\dd W(t),
\end{align*}
where $\{W(t)\}_{t\ge0}$ denotes a standard Brownian motion. As in section~\ref{ex1d}, assume that the searchers start in the left half of the interval, $x_{0}\in(0,l/2)$.

Define $0<L_{0}:=x_{0}<l-x_{0}=:L_{1}$. In the Appendix, we show that $F(t)$ satisfies \eqref{shortF} with
\begin{align}\label{driftA}
A
=\exp\Big(\frac{-{\mu}L_{0}}{2D}\Big)
\sqrt{\frac{4D}{\pi (L_{0})^{2}}},\quad 
p
=\frac{1}{2},\quad
C_{0}
=\frac{(L_{0})^{2}}{4D},
\end{align}
and that $F_{1}(t)$ satisfies \eqref{shortfar} with
\begin{align}\label{driftB}
B
=\exp\Big(\frac{{\mu}L_{1}}{2D}\Big)
\sqrt{\frac{4D}{\pi (L_{1})^{2}}},\quad 
q
=p
=\frac{1}{2},\quad
C_{1}
=\frac{(L_{1})^{2}}{4D}.
\end{align}
In particular, the short-time asymptotics of $F$ and $F_{1}$ are unchanged from the problem in section~\ref{ex1d} with zero drift except for the factor of $\exp(-{\mu}L_{0}/(2D))$ in $F$ and the factor of $\exp({\mu}L_{1}/(2D))$ in $F_{1}$.

Therefore, Theorem~\ref{detailprob} implies that \begin{align}\label{driftas}
&\P(K_{N}=1)
\sim
{{\eta}}
(\ln N)^{(\beta-1)/2}
N^{1-{\beta}}
\quad\text{as }N\to\infty,
\end{align}
where 
\begin{align*}
{\beta}
&=\Big(\frac{L_{1}}{L_{0}}\Big)^{2}
=\Big(\frac{l-x_{0}}{x_{0}}\Big)^{2}
>1,\\
{{\eta}}
&=\exp\Big(\frac{\mu l}{2D}\sqrt{\beta}\Big)\eta_{0}
>0,
\end{align*}
where $\eta_{0}:=\sqrt{\beta\pi^{\beta-1}}\Gamma(\beta)$ is the constant computed for the example with zero drift in section~\ref{ex1d}. Notice that the drift $\mu\in\R$ plays a minor role in the asymptotics of $\P(K_{N}=1)$ since it only affects the constant prefactor $\eta$ rather than the decay rate.

\begin{figure}
  \centering
             \includegraphics[width=1\textwidth]{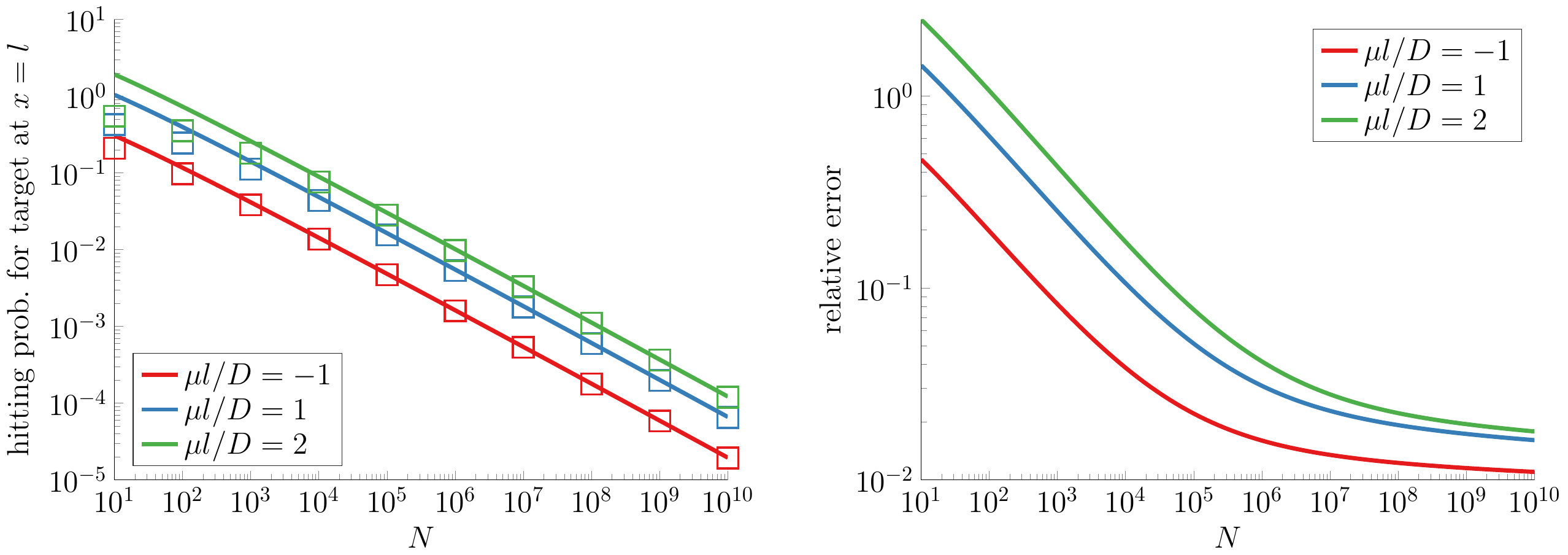}
 \caption{Extreme hitting probabilities for one-dimensional diffusion with drift. See section~\ref{drift} for details.}
 \label{fig1ddrift}
\end{figure}

In Figure~\ref{fig1ddrift}, we compare \eqref{driftas} to numerical simulations for the starting position $x_{0}=0.45 l\in(0,l/2)$. In the left panel, the solid curves are the asymptotic formula in \eqref{driftas} and the square markers are computed using numerical integration of the representation for $\P(K_{N}=1)$ given in Proposition~\ref{pint} (see the Appendix for details of the numerical method). The right panel plots the relative error.

For this example, it is straightforward to compute the probability that a given single searcher starting at $x_{0}\in(0,l)$ hits $x=l$ before $x=0$ \cite{oksendal2003},
\begin{align}\label{singledrift}
\P(K_{1}=1)
=\frac{\left(1-e^{-{{\mu}x_{0}}/{D}}\right) e^{{\mu (l-x_{0})}/{D}}}{1-e^{-{\mu l}/{D}}}.
\end{align}
For the positive values of the drift plotted in Figure~\ref{fig1ddrift} (namely, $\mu l/D=1$ and $\mu l/D=2$) and the starting position $x_{0}=0.45 l$, equation~\eqref{singledrift} implies that a given single searcher is actually more likely to hit $x=l$ before $x=0$ (i.e.\ $\P(K_{1}=1)>1/2$), despite the fact that the fastest searcher only rarely hits $x=l$ before $x=0$ if $N$ is large.

\subsection{Partially absorbing target(s)}\label{partial}

In the examples above, a target was ``found'' by the searcher as soon as the searcher touched the target. In particular, we defined the FPT to be $\tau:=\inf\{t>0:X(t)\in\cup_{k=0}^{m-1}V_{k}\}$. In this scenario, the targets are said to be ``perfectly absorbing.'' An alternative model is that of ``partially absorbing'' targets \cite{grebenkov2006, erban07}, in which the searcher ``finds'' (or ``reacts with'') a target only after spending some time near the target. Mathematically, the FPT of interest for partially absorbing targets is
\begin{align}\label{taugamma}
\tau_{\textup{partial}}
:=\inf\{t>0:\lambda_{k}(t)>\xi_{k}/\gamma_{k}\,\text{for some }k\in\{0,\dots,m-1\}\},
\end{align}
where $\{\xi_{k}\}_{k=0}^{m-1}$ are $m$ independent unit rate exponential random variables, $\{\gamma_{k}\}_{k=0}^{m-1}$ are $m$ given nonnegative parameters called ``trapping rates'' \cite{Berezhkovskii2004, lawley2019boundary}, and $\lambda_{k}(t)$ is the local time of $X(t)$ on $V_{k}$ \cite{grebenkov2006} ($\gamma_{k}>0$ has dimension length/time and $\lambda_{k}(t)$ has dimension time/length). 

Consider the example in section~\ref{ex1d}, but now suppose that the targets $V_{0}$ and $V_{1}$ have respective trapping rates $\gamma_{0}>0$ and $\gamma_{1}>0$. If we define the survival probability conditioned on the initial location of the searcher,
\begin{align}\label{pS}
S(x,t)
:=\P(\tau_{\textup{partial}}>t\,|\,X(0)=x),
\end{align}
then $S$ satisfies the backward Kolmogorov equation,
\begin{align}\label{ppde}
\frac{\partial}{\partial t}S
&=D\frac{\partial^{2}}{\partial x^{2}}S,\quad x\in(0,l),
\end{align}
with unit initial condition, $S=1$ at $t=0$, and Robin boundary conditions,
\begin{align}\label{pbcs}
\begin{split}
D\frac{\partial}{\partial x}S
&=\gamma_{0}S,\quad x=0,\\
-D\frac{\partial}{\partial x}S
&=\gamma_{1}S,\quad x=l.
\end{split}
\end{align}
Setting $\gamma_{k}=\infty$ corresponds to making $V_{k}$ perfectly absorbing, which can be seen from \eqref{pbcs} or \eqref{taugamma}.

Define $0<L_{0}:=x_{0}<l-x_{0}=:L_{1}$, where we have again assumed that the searchers start in the left half of the interval, $x_{0}\in(0,l/2)$. We conjecture that $F(t)$ satisfies \eqref{shortF} with
\begin{align}\label{ac1}
A
=\begin{cases}
\frac{2\gamma_{0}}{L_{0}}\sqrt{\frac{4D}{\pi (L_{0})^{2}}}  & \text{if }\gamma_{0}\in(0,\infty),\\
\sqrt{\frac{4D}{\pi (L_{0})^{2}}} & \text{if }\gamma_{0}=\infty,
\end{cases}
\qquad 
p
=\begin{cases}
3/2 & \text{if }\gamma_{0}\in(0,\infty),\\
1/2 & \text{if }\gamma_{0}=\infty,
\end{cases}
\end{align}
and $C_{0}=\frac{(L_{0})^{2}}{4D}$. We further conjecture that $F_{1}(t)$ satisfies \eqref{shortfar} with the analogous values of $B$, $q$, and $C_{1}$,
\begin{align}\label{ac3}
B
=\begin{cases}
\frac{2\gamma_{1}}{L_{1}}\sqrt{\frac{4D}{\pi (L_{1})^{2}}} & \text{if }\gamma_{1}\in(0,\infty),\\
\sqrt{\frac{4D}{\pi (L_{1})^{2}}} & \text{if }\gamma_{1}=\infty,
\end{cases}
\qquad 
q
=\begin{cases}
3/2 & \text{if }\gamma_{1}\in(0,\infty),\\
1/2 & \text{if }\gamma_{1}=\infty,
\end{cases}
\end{align}
and $C_{1}=\frac{(L_{1})^{2}}{4D}$. While we do not prove \eqref{ac1}-\eqref{ac3}, they can be derived by assuming that the presence of target $k$ does not affect the short-time asymptotics of $F_{1-k}$. See the Appendix for this derivation and the example in section~\ref{3d} below for a more detailed justification of an analogous conjecture in three dimensions. 

\begin{figure}
  \centering
             \includegraphics[width=1\textwidth]{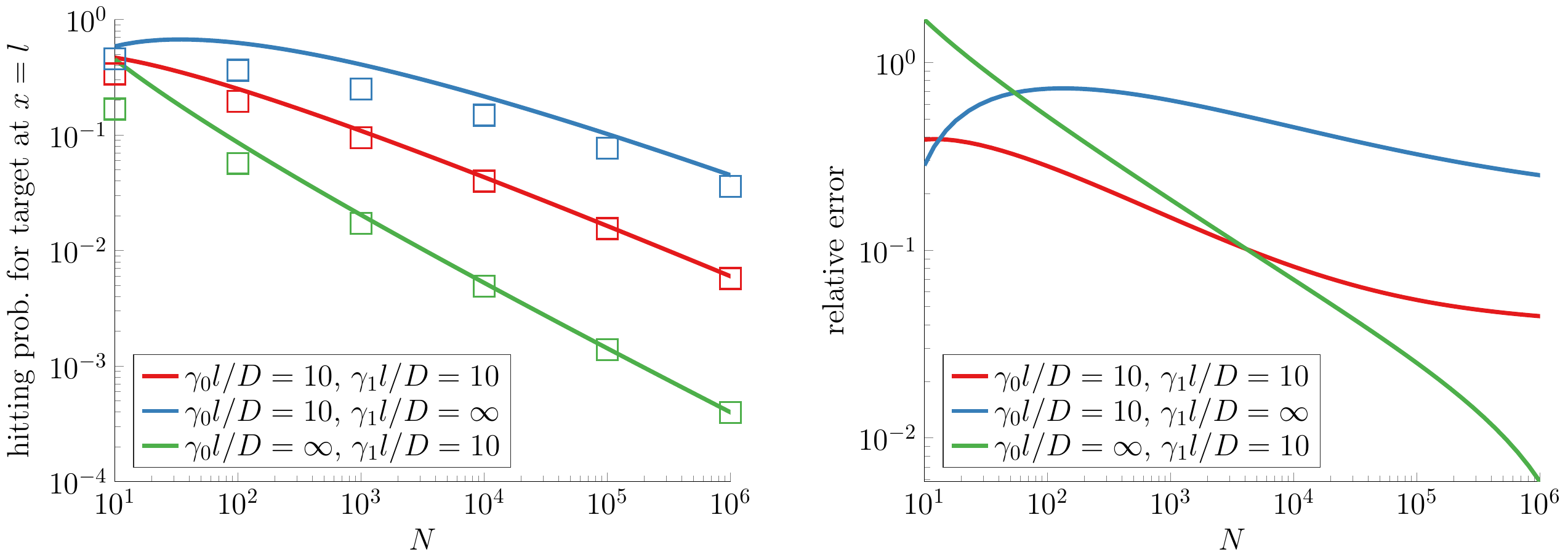}
 \caption{Extreme hitting probabilities for partially absorbing targets. See section~\ref{partial} for details.}
 \label{figpartial}
\end{figure}

Assuming \eqref{shortF} and \eqref{shortfar} hold with \eqref{ac1}-\eqref{ac3}, Theorem~\ref{detailprob} implies that
\begin{align}\label{partialas}
&\P(K_{N}=1)
\sim
{{\eta}}
(\ln N)^{\rho}
N^{1-{\beta}}
\quad\text{as }N\to\infty,
\end{align}
where ${\beta}=(\frac{L_{1}}{L_{0}})^{2}=(\frac{l-x_{0}}{x_{0}})^{2}>1$ and the values of the constant prefactor $\eta>0$ and the logarithmic power $\rho\in\R$ depend on which target(s) is partially or perfectly absorbing. Specifically, 
\begin{alignat*}{2}
\eta
&=\tfrac{2\gamma_{1}}{L_{1}}(\tfrac{2\gamma_{0}}{L_{0}})^{-\beta}(\tfrac{(L_{0})^{2}}{4D})^{1-\beta}\eta_{0},
\quad
&&\rho
=\tfrac{3}{2}\beta-\tfrac{3}{2}
\quad\text{if }\gamma_{0},\gamma_{1}\in(0,\infty),\\
\eta
&=(\tfrac{2\gamma_{0}}{L_{0}})^{-\beta}(\tfrac{(L_{0})^{2}}{4D})^{-\beta}\eta_{0},
\quad
&&\rho
=\tfrac{3}{2}\beta-\tfrac{1}{2}
\quad\text{if }\gamma_{0}\in(0,\infty),\gamma_{1}=\infty,\\
\eta
&=\tfrac{2\gamma_{1}}{L_{1}}\tfrac{(L_{0})^{2}}{4D}\eta_{0},
\quad
&&\rho
=\tfrac{1}{2}\beta-\tfrac{3}{2}
\quad\text{if }\gamma_{0}=\infty,\gamma_{1}\in(0,\infty),\\
\eta
&=\eta_{0},
\quad
&&\rho
=\tfrac{1}{2}\beta-\tfrac{1}{2}
\quad\text{if }\gamma_{0}=\gamma_{1}=\infty,
\end{alignat*}
where $\eta_{0}:=\sqrt{\beta\pi^{\beta-1}}\Gamma(\beta)$ is the constant computed for the example with perfectly absorbing targets in section~\ref{ex1d}.

In Figure~\ref{figpartial}, we compare \eqref{partialas} to numerical simulations for the starting position $x_{0}=0.45 l\in(0,l/2)$. In the left panel, the solid curves are the asymptotic formula in \eqref{partialas} and the square markers are computed using numerical integration of the representation for $\P(K_{N}=1)$ given in Proposition~\ref{pint} (see the Appendix for details of the numerical method). The right panel plots the relative error.

\subsection{Concentric targets in three dimensions}\label{3d}

Consider pure diffusion with diffusivity $D>0$ in three-dimensional space $\R^{3}$. Suppose there is an ``inner'' target at the origin with radius $R_{0}>0$,
\begin{align*}
V_{0}
:=\{x\in\R^{3}:\|x\|\le R_{0}\},
\end{align*}
and an ``outer'' target defined by
\begin{align*}
V_{1}
:=\{x\in\R^{3}:\|x\|\ge R_{1}\},
\end{align*}
where $R_{1}>R_{0}>0$ and $\|\cdot\|$ denotes the standard Euclidean norm. Suppose the searchers start at radius $\|X(0)\|\in(R_{0},R_{1})$ between these two concentric targets (see the left panel of Figure~\ref{fig3dschem} for an illustration). Suppose the searchers start closer to the inner target so that the distances to the targets satisfy
\begin{align*}
L_{0}
:=\|X(0)\|-R_{0}
<R_{1}-\|X(0)\|
=:L_{1}.
\end{align*}

\begin{figure}
  \centering
\includegraphics[width=0.3\textwidth]{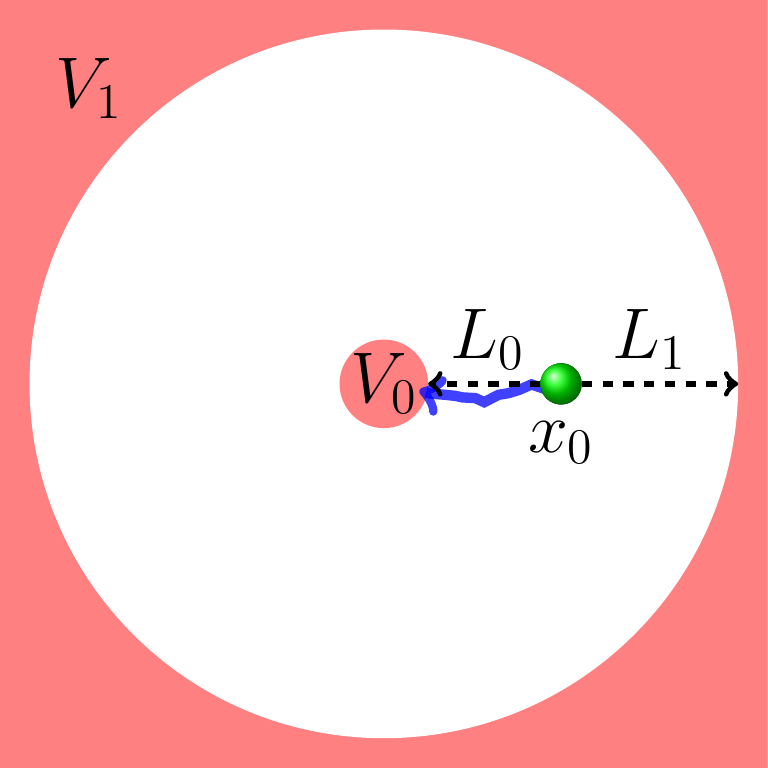}
\qquad\qquad
\includegraphics[width=0.35\textwidth]{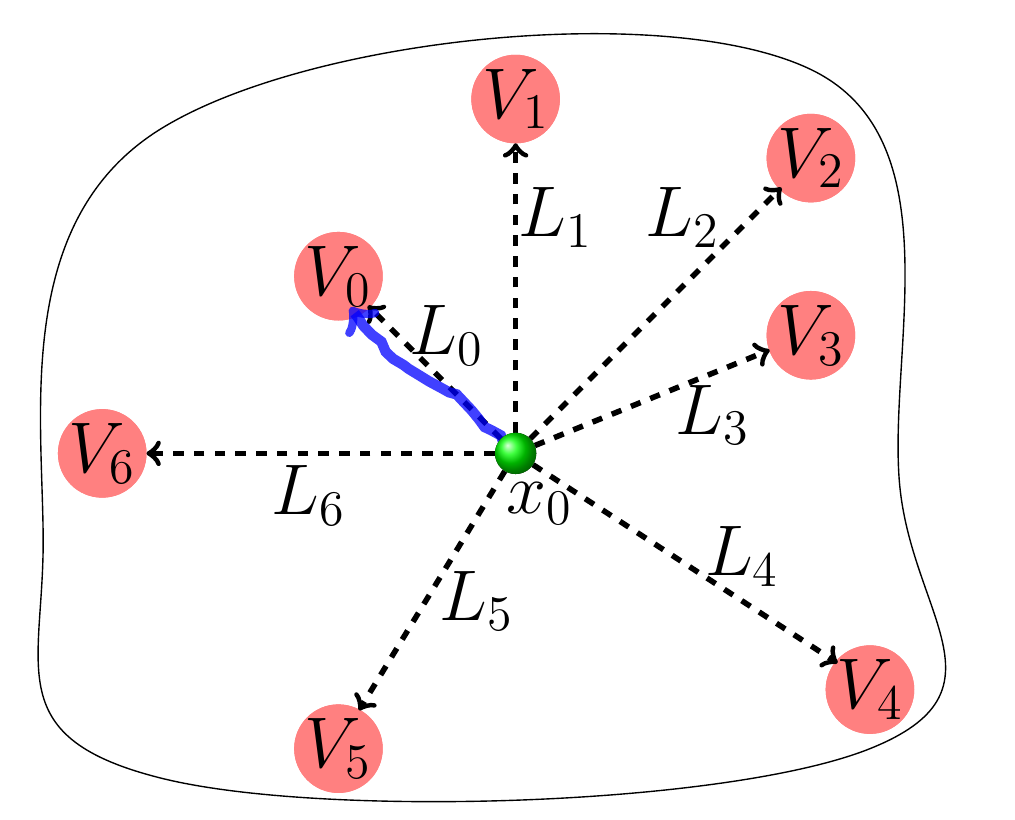}
 \caption{\textbf{Left:} Schematic diagram of concentric targets in three dimensions studied in section~\ref{3d}. \textbf{Right:} Schematic diagram of the narrow capture problem in three dimensions studied in section~\ref{narrow}.}
 \label{fig3dschem}
\end{figure}

We conjecture that
\begin{align}
F(t)
&\sim At^{p}e^{-C_{0}/t}\quad\text{as }t\to0+,\label{sc1}\\
F_{1}(t)
&\sim Bt^{q}e^{-C_{1}/t}\quad\text{as }t\to0+,\label{sc2}
\end{align}
where
\begin{align}
A
&=\frac{R_{0}}{\|X(0)\|}\sqrt{\frac{4D}{\pi(L_{0})^{2}}},\quad
p
=1/2,\quad
C_{0}
=\frac{(L_{0})^{2}}{4D},\label{sc3}\\
B
&=\frac{R_{1}}{\|X(0)\|}\sqrt{\frac{4D}{\pi(L_{1})^{2}}},\quad
q=p=1/2,\quad
C_{1}
=\frac{(L_{1})^{2}}{4D}.\label{sc4}
\end{align}
While we do not prove \eqref{sc1}-\eqref{sc3}, their informal justification is the following. The asymptotic relations in \eqref{sc1}-\eqref{sc2} concern the behavior of searchers which hit a target at an early time, and such searchers tend to follow the shortest path, which in this case is a straight line \cite{varadhan1967}. Since the straight line path from $X(0)$ to $V_{0}$ does not intersect $V_{1}$, we expect that \eqref{sc1} and \eqref{sc3} would be unchanged if $V_{0}$ was the only target. This idea is often called the ``principle of not feeling the boundary'' \cite{van1989, hsu1995}. In the case that $V_{0}$ is indeed the only target, we can solve for the distribution of $\tau$ exactly and show that it satisfies \eqref{sc1} and \eqref{sc3} (see the Appendix). Similarly, if $V_{1}$ is the only target, then we can show that the distribution of $\tau$ satisfies \eqref{sc2} and \eqref{sc4} (see the Appendix). 

Assuming \eqref{sc1}-\eqref{sc4}, then Theorem~\ref{detailprob} implies that
\begin{align}\label{3das}
&\P(K_{N}=1)
\sim
{{\eta}}
(\ln N)^{(\beta-1)/2}
N^{1-{\beta}}
\quad\text{as }N\to\infty,
\end{align}
where 
\begin{align*}
{\beta}
&=\Big(\frac{L_{1}}{L_{0}}\Big)^{2}
=\Big(\frac{R_{1}-\|X(0)\|}{\|X(0)\|-R_{0}}\Big)^{2}
>1,\\
{{\eta}}
&=\frac{R_{1}}{R_{0}}\|X(0)\|^{\beta-1}\eta_{0}
>0,
\end{align*}
where $\eta_{0}:=\sqrt{\beta\pi^{\beta-1}}\Gamma(\beta)$ is the constant prefactor computed for the example in section~\ref{ex1d}. Note that $\eta\to\eta_{0}$ if we take $R_{0}$, $R_{1}$, and $\|X(0)\|$ to infinity while keeping $L_{0}$ and $L_{1}$ fixed, which is to be expected since target curvature becomes irrelevant in this limit and the problem becomes one-dimensional.

In Figure~\ref{fig3d}, we compare \eqref{3das} to numerical simulations for the starting position $\|X(0)\|=R_{1}/2>R_{0}$. In the left panel, the solid curves are the asymptotic formula in \eqref{3das} and the square markers are computed using numerical integration of the representation for $\P(K_{N}=1)$ given in Proposition~\ref{pint} (see the Appendix for details of the numerical method). The right panel plots the relative error.

For this example, it is straightforward to compute the probability that a given single searcher starting at $\|X(0)\|\in(R_{0},R_{1})$ hits $V_{1}$ before $V_{0}$ \cite{oksendal2003},
\begin{align}\label{singleconc}
\P(K_{1}=1)
=\frac{R_{1}}{R_{1}-R_{0}}\frac{\|X(0)\|-R_{0}}{\|X(0)\|}.
\end{align}
For the values of $R_{0}$ plotted in Figure~\ref{fig1ddrift} and the starting radius $\|X(0)\|=R_{1}/2$, equation~\eqref{singleconc} implies that a given single searcher is actually more likely to hit $V_{1}$ before $V_{0}$ (i.e.\ $\P(K_{1}=1)>1/2$), despite the fact that the fastest searcher only rarely hits $V_{1}$ before $V_{0}$ if $N$ is large.

\begin{figure}
  \centering
             \includegraphics[width=1\textwidth]{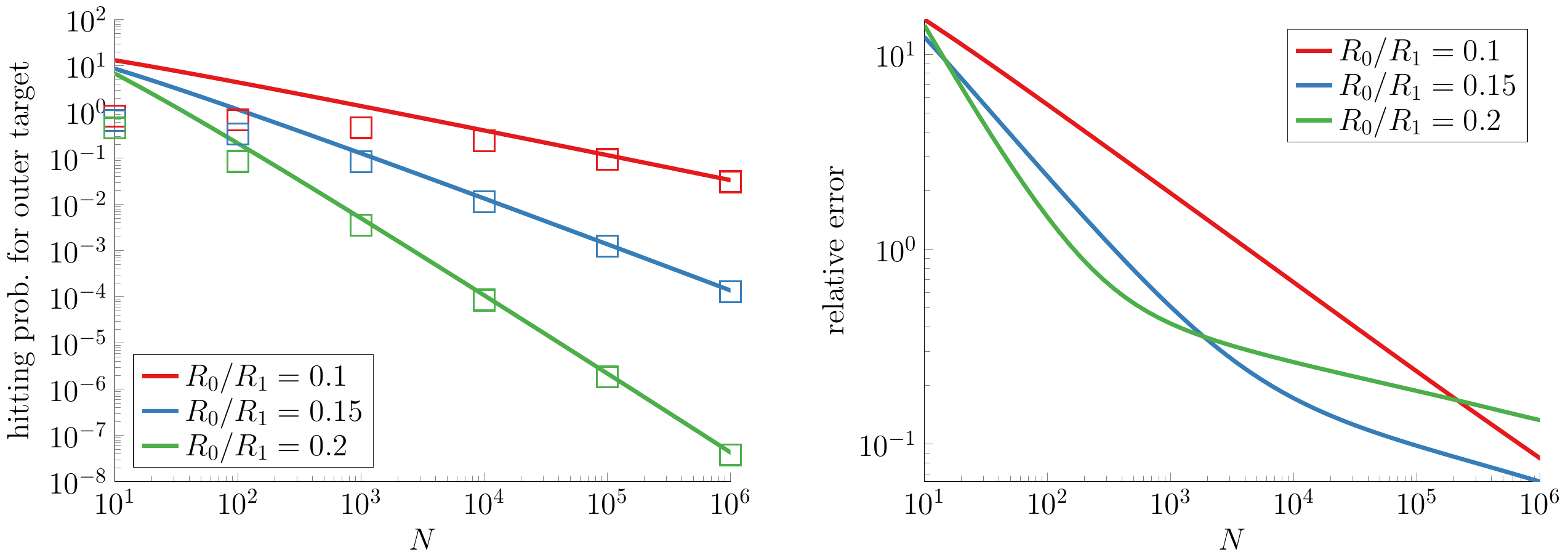}
 \caption{Extreme hitting probabilities for concentric targets in three dimensions. See section~\ref{3d} for details.}
 \label{fig3d}
\end{figure}

\subsection{Narrow capture in three dimensions}\label{narrow}

Consider pure diffusion with diffusivity $D>0$ in a bounded three-dimensional domain $M\subset\R^{3}$ with a reflecting boundary. Suppose there are $m\ge2$ small spherical targets centered at the $m$ distinct points $v_{0},\dots,v_{m-1}\in \textup{int}(M)$ with respective radii $\eps r_{0},\dots,\eps r_{m-1}>0$ for some $\eps>0$ ($\textup{int}(M)$ denotes the interior of $M$). That is, the targets are
\begin{align*}
V_{k}
:=\{x\in\R^{3}:\|x-v_{k}\|\le \eps r_{k}\},\quad k\in\{0,\dots,m-1\}.
\end{align*}
This problem is often called the narrow capture problem \cite{kaye2020}, and one studies the statistics of a single searcher in the small target limit, $\eps\to0$. See the right panel of Figure~\ref{fig3dschem} for an illustration.

Assume the searchers start at $x_{0}\notin \cup_{k=0}^{m-1}V_{k}$ and that the three points $x_{0}$, $v_{k}$, and $v_{j}$ are not collinear for any $k\neq j$. Assume that $\eps$ is sufficiently small so that (i) $V_{k}\subset M$ for each $k\in\{0,\dots,m-1\}$ and (ii)
\begin{align}\label{notcollinear}
p x_{0}+(1-p)v_{k}\notin V_{j}\quad\text{for all }p\in[0,1],\,k\neq j.
\end{align}
The assumption in~\eqref{notcollinear} ensures that the straight line path from $x_{0}$ to $V_{k}$ does not intersect any other target. Assume further that the shortest path from $x_{0}$ to each target $V_{k}$ lies entirely in the interior of $M$, 
\begin{align}\label{notboundary}
p x_{0}+(1-p)v_{k}\in \textup{int}(M)\quad\text{for all }p\in[0,1],\,k\in\{0,\dots,m-1\}.
\end{align}
Assume there is a unique closest target to $x_{0}$ and without loss of generality assume it is $V_{0}$. That is, assume 
\begin{align*}
0<L_{0}
:=\|x_{0}-v_{0}\|-\eps r_{0}
<\|x_{0}-v_{k}\|-\eps r_{k}
=:L_{k}
\quad\text{for all }k\in\{1,\dots,m-1\}.
\end{align*}

Under these assumptions, we conjecture that for $k\in\{1,\dots,m-1\}$,
\begin{align}\label{nc}
F(t)
&\sim At^{p}e^{-C_{0}/t}\quad\text{as }t\to0+,\\
F_{k}(t)
&\sim Bt^{q}e^{-C_{k}/t}\quad\text{as }t\to0+,
\end{align}
where
\begin{align}
A
&=\frac{\eps r_{0}}{\|x_{0}-v_{0}\|}\sqrt{\frac{4D}{\pi(L_{0})^{2}}},\quad
p=\frac{1}{2},\quad
C_{0}
=\frac{(L_{0})^{2}}{4D}\\
B
&=\frac{\eps r_{k}}{\|x_{0}-v_{k}\|}\sqrt{\frac{4D}{\pi(L_{k})^{2}}},\quad
q=p=\frac{1}{2},\quad
C_{k}
=\frac{(L_{k})^{2}}{4D}.\label{nc2}
\end{align}
We do not prove \eqref{nc}-\eqref{nc2}, but their derivation is analogous to the derivation of \eqref{sc1} and \eqref{sc3} (i.e.\ one finds the short-time asymptotics of $F_{k}$ assuming $V_{k}$ is the only target). 

Assuming \eqref{nc}-\eqref{nc2}, Theorem~\ref{detailprob} implies that for $k\in\{1,\dots,m-1\}$,
\begin{align}\label{ck}
&\P(K_{N}=k)
\sim
{{\eta}}
(\ln N)^{(\beta-1)/2}
N^{1-{\beta}}
\quad\text{as }N\to\infty,
\end{align}
where 
\begin{align*}
{\beta}
&=\Big(\frac{L_{k}}{L_{0}}\Big)^{2}
=\Big(\frac{\|x_{0}-v_{k}\|-\eps r_{k}}{\|x_{0}-v_{0}\|-\eps r_{0}}\Big)^{2}
>1,\\
{{\eta}}
&=\frac{\eps r_{k}}{\|x_{0}-v_{k}\|}\Big(\frac{\eps r_{0}}{\|x_{0}-v_{0}\|}\Big)^{\beta}\eta_{0}
>0,
\end{align*}
where $\eta_{0}:=\sqrt{\beta\pi^{\beta-1}}\Gamma(\beta)$ is the constant prefactor computed for the one-dimensional example in section~\ref{ex1d}.

It is interesting to contrast \eqref{ck} with the behavior of $\P(K_{N}=k)$ in the small target limit, $\eps\to0$. In the case of a single searcher searching for small spherical targets, the probability it hits a particular target is merely the ratio of the target radii \cite{cheviakov11},
\begin{align}\label{prevr}
\P(K_{1}=k)
\to\frac{r_{k}}{\sum_{j=0}^{m-1}r_{j}}\quad\text{as }\eps\to0.
\end{align}
The intuitive reason for \eqref{prevr} is that in the small target limit (i.e.\ $\eps\to0$), the searcher wanders around the entire domain before finding a target and thus the probability it hits any particular target depends merely on the target sizes. In particular, notice that the limit in \eqref{prevr} is independent of the starting location $x_{0}$ (assuming $x_{0}$ is outside an order $\eps$ neighborhood of each target, which is true if $x_{0}$ is fixed and $\eps\to0$). We conjecture that the limit in \eqref{prevr} actually holds for any fixed $N\ge2$, 
\begin{align}\label{nr}
\P(K_{N}=k)
\to\frac{r_{k}}{\sum_{j=0}^{m-1}r_{j}}\quad\text{as }\eps\to0.
\end{align}
The intuitive reasoning behind \eqref{nr} is the same as \eqref{prevr}. Namely, in the small target limit, even the fastest searcher wanders around the entire domain before finding the target.

Therefore, the many searcher limit $N\to\infty$ and the small target limit $\eps\to0$ constitute competing limits. It would be interesting to understand the crossover regime between small $\eps$ and large $N$. An analysis of similar competing limits between many searchers and small targets was carried out for extreme FPTs in \cite{madrid2020comp}.

\section{General diffusion processes}\label{general}

In the examples above, we used Theorem~\ref{detailprob} to calculate the exact asymptotics of the distribution of $K_{N}$ as $N\to\infty$. We were able to find these exact asymptotics because the specifics of the examples allowed us to obtain the detailed short-time behavior of $F(t)$ and $F_{k}(t)$.

In the case of more complicated geometries or more complicated diffusion processes, this detailed short-time behavior of $F(t)$ and $F_{k}(t)$ is not available. However, we are able to obtain bounds on the short-time behavior of $F(t)$ and $F_{k}(t)$ on a logarithmic scale in significant generality. In particular, under very general assumptions, it is known that
\begin{align}
\lim_{t\to0+}t\ln F(t)
&=-\frac{(L_{0})^{2}}{4D}<0,\label{lg1}\\
\lim_{t\to0+}t\ln F_{k}(t)
&\le-\frac{(L_{k})^{2}}{4D}<0,\quad k\in\{1,\dots,m-1\},\label{lg2}
\end{align}
where $L_{k}>L_{0}>0$ are certain geodesic distances from the set of starting locations to the targets. Hence, we can apply Theorem~\ref{logprob} to obtain an upper bound on the asymptotics of the distribution of $K_{N}$. In particular, Theorem~\ref{logprob} implies that for any $\eps>0$,
\begin{align}\label{ubg}
\P(K_{N}=k)
=o(N^{1-(L_{k}/L_{0})^{2}+\eps})\quad\text{as }N\to\infty.
\end{align}

The point of this section is to show some of the general scenarios in which we can conclude that \eqref{ubg} holds because \eqref{lg1}-\eqref{lg2} hold and to show the values of the geodesic lengths $L_{0}$ and $L_{k}$. Our approach in this section adapts the analysis in \cite{lawley2020uni, lawley2021mortal} which established \eqref{lg1} in order to study extreme FPTs.

\subsection{Setup}

Let $\{X(t)\}_{t\ge0}$ be a ${{d}}$-dimensional diffusion process (i.e.\ the ``searcher'') on a manifold $M$ that contains $m\ge2$ pairwise disjoint ``targets'' denoted by $V_{0},\dots,V_{m-1}$. For each $k\in\{0,\dots,m-1\}$, assume $V_{k}\subset M$ is the closure of its interior which precludes trivial cases such as a target being a single point. Assume the initial distribution of $X$ has compact support $U_{0}\subset M$ that does not intersect any target,
\begin{align}\label{away}
U_{0}\cap V_{k}=\varnothing,\quad\text{for each }k\in\{0,\dots,m-1\}.
\end{align}

Suppose we are given a distance function between points in $M$,
\begin{align}\label{gl}
L:M\times M\to[0,\infty).
\end{align}
Let $L_{k}$ denote the shortest distance from the starting locations $U_{0}$ to the $k$th target $V_{k}\subset M$,
\begin{align}\label{dset0}
L_{k}
=\inf_{x_{0}\in U_{0},{{x}}\in V_{k}}L({x_{0}},{{x}})>0,\quad k\in\{0,\dots,m-1\}.
\end{align}
Assume that there is a unique closest target, which we take to be $V_{0}$ without loss of generality. That is, assume
\begin{align*}
0<L_{0}<L_{k}\quad\text{for all }k\in\{1,\dots,m-1\}.
\end{align*}

Let $\tau^{(k)}$ denote the FPT to the $k$th target,
\begin{align}\label{tauk9}
\tau^{(k)}
:=\inf\{t>0:X(t)\in V_{k}\}, 
\end{align}
and let $\tau$ denote the FPT to any of the targets,
\begin{align}\label{tau}
\tau
:=\min_{k\in\{0,\dots,m-1\}}\tau^{(k)}
=\inf\{t>0:X(t)\in \cup_{k=0}^{m-1}V_{k}\}.
\end{align}
Hence, $\tau\le\tau^{(k)}$, and therefore
\begin{align}\label{bound1}
F(t)
&:=\P(\tau\le t)
\ge\P(\tau^{(k)}\le t)\quad\text{for any }k\in\{0,\dots,m-1\}.
\end{align}
Furthermore, $\tau=\tau^{(k)}$ if $\kappa=k$ (recall from section~\ref{math} that $\kappa\in\{0,\dots,m-1\}$ denotes the index of the target hit by the searcher), and therefore
\begin{align}\label{bound2}
F_{k}(t)
&:=\P(\tau\le t\cap \kappa=k)
=\P(\tau^{(k)}\le t\cap \kappa=k)
\le\P(\tau^{(k)}\le t).
\end{align}
In the examples below, we show that
\begin{align}\label{suff}
\lim_{t\to0+}t\ln\P(\tau^{(k)}\le t)
&=-\frac{(L_{k})^{2}}{4D}<0,
\end{align}
for an appropriately chosen distance function $L$ in \eqref{gl}. Therefore, once \eqref{suff} is established, Theorem~\ref{logprob} and the bounds in \eqref{bound1}-\eqref{bound2} yield \eqref{ubg}.

\subsection{Pure diffusion in $\R^{d}$}

Consider first the case of pure diffusion in $M=\R^{{d}}$ with diffusivity $D>0$. It was shown in \cite{lawley2020uni} that \eqref{suff} holds with the distance function in \eqref{gl} given by the standard Euclidean length, $L=\deuc$,
\begin{align}\label{euc}
\deuc({x_{0}},{{x}})
:=\|x_{0}-x\|,\quad x_{0},x\in\R^{d}.
\end{align}
We therefore conclude by Theorem~\ref{logprob} and \eqref{bound1}-\eqref{bound2} that \eqref{ubg} holds with the Euclidean length \eqref{euc}.

\subsection{Space-dependent diffusivity and drift in $\R^{d}$}

Rather than pure diffusion, assume the searcher moves according to the following It\^{o} stochastic differential equation on $M=\R^{{d}}$,
\begin{align}\label{sde}
\begin{split}
\dd X
&={{{\mu}}}(X)\,\dd t+\sqrt{2D}{\sigma}(X)\,\dd W,
\end{split}
\end{align}
where ${{{\mu}}}:\R^{{d}}\to\R^{{d}}$ is a space-dependent drift that describes any deterministic forces on the searcher, $D>0$ is a characteristic diffusion coefficient, ${\sigma}:\R^{{d}}\to\R^{{{d}}\times r}$ is a dimensionless, matrix-valued function that describes any anisotropy or space-dependence in the diffusivity, and $W(t)\in\R^{r}$ is a standard Brownian motion in $r$-dimensional space. Following \cite{lawley2020uni}, we assume that $\R^{{d}}\backslash \cup_{k=0}^{m-1}V_{k}$ is bounded and we make the following technical assumptions on the coefficients in \eqref{sde}: ${{{\mu}}}$ is uniformly bounded and uniformly Holder continuous and ${\sigma}{\sigma}^{\top}$ is uniformly Holder continuous and its eigenvalues are in a finite interval $(\nu_{1},\nu_{2})$ with $\nu_{1}>0$.

For any smooth path $\omega:[0,1]\to M$, define its length, $l(\omega)$, in the following Riemannian metric which depends on the inverse of the diffusion matrix in \eqref{sde}, $a:={\sigma}{\sigma}^{\top}$,
\begin{align}\label{ll}
l(\omega)
:=\int_{0}^{1}\sqrt{\dot{\omega}^{\top}(s)a^{-1}(\omega(s))\dot{\omega}(s)}\,\dd s.
\end{align}
For any two points $x_{0},x\in\R^{d}$, define the geodesic length between the points to be the following infimum of $l(\omega)$ over all smooth paths $\omega:[0,1]\to M$ which connect $\omega(0)={x_{0}}$ to $\omega(1)={{x}}$:
\begin{align}\label{drie}
\begin{split}
\drie({x_{0}},{{x}})
&:=\inf\{l(\omega):\omega(0)=x_{0},\,\omega(1)=x\},\quad x_{0},x\in\R^{d}.
\end{split}
\end{align}
Under these assumption, Varadhan's formula \cite{varadhan1967} was used in \cite{lawley2020uni} to show that \eqref{suff} holds with distance function in \eqref{gl} given by $L=\drie$. We therefore conclude by Theorem~\ref{logprob} that \eqref{ubg} holds with the length \eqref{drie}.

We emphasize two points about this result. First, the bound in \eqref{ubg} on the decay of the extreme hitting probabilities is independent of the drift. To see this, note that the distance function $\drie$ in \eqref{drie} does not depend on the drift $\mu(X)$ in \eqref{sde}. Hence, the target distances $L_{0}$ and $L_{k}$ appearing in the bound in \eqref{ubg} are computed without any consideration of the drift. This accords with the one-dimensional example with constant drift considered in section~\ref{drift}, where we found that the drift affects only the constant prefactor in the asymptotic behavior of the extreme hitting probability.

Second, the bound in \eqref{ubg} on the decay of the extreme hitting probabilities does depend on $\sigma(X)$, which describes the space-dependence or anisotropy in the diffusion. In particular, notice that the length function $l(\omega)$ in \eqref{ll} penalizes paths which traverse regions of slow diffusivity. Hence, the distances $L_{0}$ and $L_{k}$ appearing in the bound in \eqref{ubg} are the lengths of the shortest paths to the targets which avoid regions of slow diffusivity.

\subsection{Diffusion on a manifold with reflecting obstacles}\label{manifold}

Assume $M$ is a ${{d}}$-dimensional smooth Riemannian manifold. As one example, $M$ could be a set in $\R^{d}$ with smooth boundaries which model reflecting obstacles, as illustrated in Figure~\ref{figschem0}. Assume $\{X(t)\}_{t\ge0}$ is a diffusion on $M$ which is described by its generator $\L$, which in each coordinate chart is a second order differential operator of the following form
\begin{align*}
\L f
=D\sum_{i,j=1}^{n}\frac{\partial}{\partial x_{i}}\bigg(a_{ij}(x)\frac{\partial f}{\partial x_{j}}\bigg),
\end{align*}
where $a=\{a_{ij}\}_{i,j=1}^{n}$ satisfies some mild technical conditions (namely, in each coordinate chart, assume $a$ is continuous, symmetric, and that its eigenvalues are in a finite interval $(\nu_{1},\nu_{2})$ with $\nu_{1}>0$). Assume $M$ is connected and compact and assume that $X$ reflects from the boundary of $M$ if $M$ has a boundary.

Relying on the results of \cite{norris1997}, it was shown in \cite{lawley2020uni} that \eqref{suff} holds with distance function in \eqref{gl} given by $L=\drie$ in \eqref{drie}. We therefore again conclude by Theorem~\ref{logprob} that \eqref{ubg} holds with the length \eqref{drie} for this example. For the example of diffusion in the presence of reflecting obstacles as illustrated in Figure~\ref{figschem0}, we point out that the lengths $L_{0}$ and $L_{k}$ in the bound in \eqref{ubg} are the lengths of the shortest paths to the targets which go around the obstacles (recall that the infimum in \eqref{drie} is taken over paths $\omega$ lying in $M$, and therefore paths $\omega$ that intersect obstacles are prohibited).

\subsection{Partially absorbing targets}

In section~\ref{partial}, we considered partially absorbing targets in a one-dimensional example. We now consider partially absorbing targets in a more general setting. Specifically, consider pure diffusion with diffusivity $D>0$ in a smooth bounded domain in $\R^{d}$ where the target is any finite disjoint union of hyperspheres. Let $\tau_{\textup{partial}}^{(k)}$ be the FPT for the searcher to be absorbed at $V_{k}$ in the case that $V_{k}$ is partially absorbing,
\begin{align*}
\tau_{\textup{partial}}^{(k)}
:=\inf\{t>0:\lambda_{k}(t)>\xi_{k}/\gamma_{k}\}, \quad k\in\{0,\dots,m-1\},
\end{align*}
where $\lambda_{k}(t)$ is the local time of $X(t)$ on $V_{k}$, $\xi_{k}$ is an independent unit rate exponential random variable, and $\gamma_{k}>0$ is a given parameter (the so-called ``trapping rate'' of the $k$th target \cite{Berezhkovskii2004, lawley2019boundary}). In this case, it is known that \cite{lawley2020uni}
\begin{align*}
\lim_{t\to0+}t\ln\P(\tau_{\textup{partial}}^{(k)}\le t)
=\lim_{t\to0+}t\ln\P(\tau^{(k)}\le t)
=-\frac{(L_{k})^{2}}{4D}<0,
\end{align*}
where the distance function~\eqref{gl} is the standard Euclidean distance in \eqref{euc}.

We therefore conclude by Theorem~\ref{logprob} that \eqref{ubg} holds with the length \eqref{drie}. In particular, the fact that the targets are partially absorbing rather than perfectly absorbing has no effect on the bound in \eqref{ubg}. This result accords with the one-dimensional example in section~\ref{partial}, where we found that making the targets partially absorbing affects only the constant prefactor and the logarithmic power in the asymptotic behavior of the extreme hitting probability.

\section{Discussion}

In this paper, we studied extreme hitting probabilities for diffusive search in the many searcher limit. Our results yield the exact asymptotics of these extreme hitting probabilities in terms of the short-time asymptotics of the hitting time of a single searcher. We illustrated these results in several examples and numerical simulations. We also proved a general bound on the extreme hitting probabilities in terms of the distances that the searcher must travel to hit the targets.

To our knowledge, the only other work that considers what we call extreme hitting probabilities is the very interesting 2015 work of Krapivsky, Majumdar, Rosso \cite{krapivsky2010}. These authors consider $N$ purely diffusive searchers on the positive real line and study the tail of the probability distribution of the position of the searcher farthest from the origin at the time when the first searcher hits the origin. Their approach involves computing the extreme hitting probabilities for the example we considered in section~\ref{ex1d}, with the additional complication that the $N$ searchers can start at $N$ specified locations. Using that the distribution of the position of a single searcher can be written as an infinite series, the authors provide an exact representation for the extreme hitting probabilities in terms of $N$ nested infinite summations. It is not clear to us how to derive the large $N$ behavior of the extreme hitting probabilities from their novel representation.

The present work is related to several recent studies of extreme FPTs, which is the time it takes the fastest searcher to find a target out of many searchers. Extreme FPTs for diffusive search were first studied in 1983 by Weiss, Shuler, and Lindenberg \cite{weiss1983}. Driven primarily by applications to cell biology, extreme FPTs for diffusion have been recently studied by several groups of authors \cite{dybiec_escape_2015, lawley2020esp1, lawley2020uni, lawley2020dist, madrid2020comp, ro2017, clementi2020, basnayake2019, schuss2019 ,godec2016x, hartich2018, hartich2019}. Extreme FPTs for other types of search processes (i.e.\ non-diffusive) were considered in \cite{lawley2020sub, weng2017, feinerman2012, lawley2020networks, lawley2021pdmp, lawley2021super}.

In closing, the present work highlights how the behavior of a given single searcher is vastly different than the behavior of the fastest searcher out of many searchers. Furthermore, we have shown that analyzing the fastest searcher can in fact be much simpler than analyzing a single searcher. Indeed, details of the problem which are critical for a single searcher (domain size, domain geometry, spatial dimension, drift, etc.)\ are irrelevant for the fastest searcher, as the extreme hitting probabilities are primarily determined simply by the target distances. Moreover, while the behavior of a single searcher may be essentially unpredictable (the searcher could be equally likely to hit each of the $m\ge2$ targets), the fastest searcher becomes effectively deterministic for many searchers, as it hits the closest target with high probability.

\section{Appendix}

In this appendix, we first give the proofs of the propositions and theorems and then give details on the numerical methods.

\subsection{Proofs}\label{proofs}

\begin{proof}[Proof of Proposition~\ref{pint}]
The result $\P(K_{N}=\infty)=(\P(\tau=\infty))^{N}$ is immediate. Let $k\in\{0,\dots,m-1\}$. Since $\{(\tau_{n},\kappa_{n})\}_{n\ge1}$ are identically distributed, we have that
\begin{align}\label{start0}
\begin{split}
\P(K_{N}=k)
=\sum_{n=1}^{N}\P(\tau_{n}=T_{N}\cap\kappa_{n}=k)
&=N\P(\tau_{N}=T_{N}\cap\kappa_{N}=k)\\
&=N\P(\tau_{N}<T_{N-1}\cap\kappa_{N}=k),
\end{split}
\end{align}
where $T_{N-1}:=\min\{\tau_{1},\dots,\tau_{N-1}\}$. 
If we define
\begin{align*}
\tau_{N}^{(k)}
=\begin{cases}
\tau_{N} & \text{if }\kappa_{N}=k,\\
+\infty & \text{if }\kappa_{N}\neq k,
\end{cases}
\end{align*}
then $\P(\tau_{N}<T_{N-1}\cap\kappa_{N}=k)=\P(\tau_{N}^{(k)}<T_{N-1})$ and so \eqref{start0} can be written as
\begin{align}\label{ff9}
\P(K_{N}=k)
=N\P(\tau_{N}^{(k)}<T_{N-1}).
\end{align}

Since $\{(\tau_{n},\kappa_{n})\}_{n\ge1}$ are iid, the survival probability of $T_{N-1}$ is
\begin{align}\label{d1}
\P(T_{N-1}>t)
=(1-F(t))^{N-1},\quad t\in\R.
\end{align}
Further, the cumulative distribution function of $\tau_{N}^{(k)}$ is
\begin{align}\label{d2}
\P(\tau_{N}^{(k)}\le t)
=F_{k}(t),\quad t\in\R.
\end{align}
Now, if $X$ and $Y$ are independent random variables with $F_{X}(x):=\P(X\le x)$ and $S_{Y}(y):=\P(Y> y)$, then
\begin{align}\label{XY}
\P(X<Y)
=\E[S_{Y}(X)]
=\int_{-\infty}^{\infty}
S_{Y}(x)\,\dd F_{X}(x).
\end{align}
Combining \eqref{XY} with \eqref{ff9}-\eqref{d2} completes the proof.
\end{proof}
\begin{proof}[Proof of Proposition~\ref{p1}]
Let $\delta_{0}>0$ be such that $At^{p}e^{-C/t}<1$ and $At^{p}e^{-C/t}$ is monotonically increasing for all $t\in(0,\delta_{0}]$. Let $\delta\in(0,\delta_{0}]$, and observe that
\begin{align*}
\int_{\delta}^{\delta_{0}}
t^{q-2}e^{-C_{+}/t}\big(1-At^{p}e^{-C/t}\big)^{N-1}\,\dd t
\le\big(1-A\delta^{p}e^{-C/\delta}\big)^{N-1}
\int_{\delta}^{\delta_{0}}
t^{q-2}e^{-C_{+}/t}\,\dd t.
\end{align*}
Hence, 
\begin{align*}
I_{0,\delta}:=\int_{0}^{\delta}
t^{q-2}e^{-C_{+}/t}\big(1-At^{p}e^{-C/t}\big)^{N-1}\,\dd t
\sim I_{0,\delta_{0}}\quad\text{as }N\to\infty,
\end{align*}
as long as $I_{0,\delta}$ vanishes slower than exponentially fast as $N\to\infty$, which we prove below. The upshot is that the large $N$ behavior of $I_{0,\delta}$ is independent of $\delta$.

Changing variables $t'=t/C$ yields
\begin{align*}
I_{0,\delta}
&=\int_{0}^{\delta}t^{q-2}e^{-C_{+}/t}\big(1-At^{p}e^{-C/t}\big)^{N-1}\,\dd t\\
&=C\int_{0}^{\delta/C}C^{q-2}(t')^{q-2}e^{-(C_{+}/C)/t'}\big(1-AC^{p}(t')^{p}e^{-1/t'}\big)^{N-1}\,\dd t'.
\end{align*}
Hence, if we let 
\begin{align*}
{\beta}
:=C_{+}/C>1,
\quad
A':=AC^{p}>0,
\quad
\delta'
:=\delta/C,
\end{align*}
then it suffices to study
\begin{align*}
I'
:=\frac{I_{0,\delta}}{C^{q-1}}
=\int_{0}^{\delta'}t^{q-2}e^{-{\beta}/t}\big(1-A't^{p}e^{-1/t}\big)^{N-1}\,\dd t.
\end{align*}

It is straightforward to verify that
\begin{align}\label{logbounds}
-x(1+x)
\le\ln(1-x)
\le-x,\quad \text{for all }x\in[0,1/2].
\end{align}
Since we may write $I'$ in the form,
\begin{align*}
I'
=&\int_{0}^{\delta'}t^{q-2}\exp\big(-{\beta}/t+(N-1)\ln(1-A't^{p}e^{-1/t})\big)\,\dd t,
\end{align*}
taking $\delta$ sufficiently small so that $A't^{p}e^{-1/t}\le1/2$ for all $t\in(0,\delta']$ and using \eqref{logbounds} yields the bounds
\begin{align}\label{Iplus}
\begin{split}
I_{-}:=&\int_{0}^{\delta'}t^{q-2}\exp\Big(-{\beta}/t-(N-1)A't^{p}{e^{-1/t}}\big(1+A't^{p}{e^{-1/t}}\big)\Big)\,\dd t
\le
I'\\
&\quad\le\int_{0}^{\delta'}t^{q-2}\exp\big(-{\beta}/t-(N-1)A't^{p}{e^{-1/t}}\big)\,\dd t=:I_{+}(A').
\end{split}
\end{align}
Furthermore, since $At^{p}e^{-1/t}$ is monotonically increasing for $t\in(0,\delta']$, we have the lower bound
\begin{align}\label{Iminus}
\begin{split}
I_{-}
\ge\int_{0}^{\delta'}t^{q-2}\exp\Big(-{\beta}/t-(N-1)A't^{p}{e^{-1/t}}\big(1+A'(\delta')^{p}{e^{-1/\delta'}}\big)\Big)\,\dd t\\
=I_{+}\big(A'(1+A'(\delta')^{p}{e^{-1/\delta'}})\big),
\end{split}
\end{align}
where $I_{+}(\cdot)$ is defined in \eqref{Iplus}. 

To study $I_{+}(A_{0})$ for an arbitrary $A_{0}>0$, we change the integration variable to
\begin{align*}
u
=t^{p}e^{-1/t}.
\end{align*}
We can invert this equation to write $t$ in terms of $u$ as
\begin{align*}
t
=\frac{1}{g(u)}
:=\begin{cases}
(pW_{0}(p^{-1}u^{-1/p}))^{-1} & \text{if }p>0,\\
(pW_{-1}(p^{-1}u^{-1/p}))^{-1} & \text{if }p<0,\\
(\ln(u^{-1}))^{-1} & \text{if }p=0,
\end{cases}
\end{align*}
where $W_{0}(z)$ denotes the principal branch of the LambertW function and $W_{-1}(z)$ denotes the lower branch \cite{corless1996}. Therefore,
\begin{align*}
\dd u
=u(pt^{-1}+t^{-2})\,\dd t
=u\big(pg(u)+(g(u))^{2}\big)\,\dd t,
\end{align*}
and
\begin{align*}
I_{+}(A_{0})
&=\int_{0}^{\delta'}t^{q-2}(e^{-1/t})^{{\beta}}\exp\big(-(N-1)A_{0}t^{p}{e^{-1/t}}\big)\,\dd t\\
&=\int_{0}^{\delta''}
h(u)
\exp(-(N-1)A_{0}u)
\,\dd u,
\end{align*}
where we have set
\begin{align*}
h(u):=\big(g(u)\big)^{p{\beta}-q}\frac{g(u)}
{p+g(u)}
u^{\beta-1},\quad
\delta'':=(\delta')^{p}e^{-1/\delta'}.
\end{align*}

Using standard results on the asymptotics of the LambertW function \cite{corless1996}, it is straightforward to check that $h(u)$ has the following logarithmic singularity at the origin,
\begin{align}\label{logsing}
h(u)
\sim u^{\beta-1}(\ln(u^{-1}))^{p\beta-q}\quad\text{as }u\to0+.
\end{align}
We can thus apply Theorem~5 in \cite{bleistein1977}, which generalizes Watson's lemma to functions with logarithmic singularities of the form \eqref{logsing}, to conclude that
\begin{align}\label{A0}
I_{+}(A_{0})
\sim (A_{0})^{-\beta}\Gamma(\beta)N^{-\beta}(\ln N)^{p\beta-q}\quad\text{as }N\to\infty.
\end{align}

Therefore, combining \eqref{A0} with the bounds in \eqref{Iplus}-\eqref{Iminus} yields
\begin{align}\label{infsup}
\begin{split}
\big(1+A'(\delta')^{p}{e^{-1/\delta'}}\big)^{-\beta}
&\le
\liminf_{N\to\infty}\frac{I'}{(A')^{-\beta}\Gamma(\beta)N^{-\beta}(\ln N)^{p\beta-q}}\\
&\le\limsup_{N\to\infty}\frac{I'}{(A')^{-\beta}\Gamma(\beta)N^{-\beta}(\ln N)^{p\beta-q}}
\le1
\end{split}
\end{align}
Since the lower bound in \eqref{infsup} can be made arbitrarily close to unity be taking $\delta'$ small, and since the large $N$ behavior of $I'$ is independent of $\delta'\in(0,\delta_{0}]$, we conclude that
\begin{align*}
I'
\sim (A')^{-\beta}\Gamma(\beta)N^{-\beta}(\ln N)^{p\beta-q}\quad\text{as }N\to\infty.
\end{align*}
Recalling the relation $I_{0,\delta}=C^{q-1}I'$ and $A'=AC^{p}$ completes the proof.
\end{proof}

\begin{proof}[Proof of Theorem~\ref{detailprob}]
Define
\begin{align*}
I_{a,b}
:=\int_{a}^{b}\big(1-F(t)\big)^{N-1}\,\dd F_{1}(t).
\end{align*}
Let $\eps\in(0,1)$. By the assumptions in \eqref{s1}-\eqref{s2}, there exists a $\delta>0$ so that
\begin{align}
A_{-\eps}t^{p}e^{-C_{0}/t}
&\le F(t)
\le A_{+\eps}t^{p}e^{-C_{0}/t} \quad \text{for all }t\in(0,\delta),\label{lb199}\\
B_{-\eps}t^{q}e^{-C_{k}/t}
&\le F_{1}(t)
\le B_{+\eps}t^{q}e^{-C_{k}/t} \quad \text{for all }t\in(0,\delta),\label{lb299}
\end{align}
where $A_{\pm\eps}:=A(1\pm\eps)$ and $B_{\pm\eps}:=B(1\pm\eps)$. 
Using \eqref{lb199} and integrating by parts yields
\begin{align}\label{ibp0}
\begin{split}
&I_{0,\delta}
\le \int_{0}^{\delta}\big(1-A_{-\eps}t^{p}e^{-C_{0}/t}\big)^{N-1}\,\dd F_{1}(t)\\
&= F_{1}(\delta)\big(1-A_{-\eps}\delta^{p}e^{-C_{0}/\delta}\big)^{N-1}\\
&
+ (N-1)\int_{0}^{\delta}(pt^{-1}+C_{0}t^{-2})A_{-\eps}t^{p}e^{-C_{0}/t}F_{1}(t)\big(1-A_{-\eps}t^{p}e^{-C_{0}/t}\big)^{N-2}\,\dd t.
\end{split}
\end{align}
The first term in the righthand side of \eqref{ibp0} vanishes exponentially fast as $N\to\infty$. To bound the second term, we note that \eqref{lb299} implies that
\begin{align}
&\int_{0}^{\delta}(pt^{-1}+C_{0}t^{-2})A_{-\eps}t^{p}e^{-C_{0}/t}F_{1}(t)\big(1-A_{-\eps}t^{p}e^{-C_{0}/t}\big)^{N-2}\,\dd t\nonumber\\
&\quad\le\int_{0}^{\delta}(pt^{-1}+C_{0}t^{-2})A_{-\eps}B_{+\eps}t^{p+q}e^{-(C_{0}+C_{k})/t}\big(1-A_{-\eps}t^{p}e^{-C_{0}/t}\big)^{N-2}\,\dd t.\label{to99}
\end{align}
Using Proposition~\ref{p1} to find the large $N$ behavior of \eqref{to99} and using \eqref{ibp0} and the fact that $I_{\delta,\infty}$ vanishes exponentially fast as $N\to\infty$ yields
\begin{align*}
\limsup_{t\to\infty}\frac{I_{0,\infty}}{{{\eta}}(\ln N)^{p{\beta}-q}N^{-{\beta}}}
\le \frac{(1+\eps)}{(1-\eps)^{\beta}}.
\end{align*}
The analogous argument yields the lower bound
\begin{align*}
\liminf_{t\to\infty}\frac{I_{0,\infty}}{{{\eta}}(\ln N)^{p{\beta}-q}N^{-{\beta}}}
\ge \frac{(1-\eps)}{(1+\eps)^{\beta}}.
\end{align*}
Since $\eps\in(0,1)$ is arbitrary, and since Proposition~\ref{pint} implies that $\P(K_{N}=1)=NI_{0,\infty}$, the proof is complete.
\end{proof}


\begin{proof}[Proof of Theorem~\ref{logprob}]
Define
\begin{align*}
I_{a,b}
:=\int_{a}^{b}\big(1-F(t)\big)^{N-1}\,\dd F_{1}(t).
\end{align*}
Let $\eps>0$. By \eqref{lb}, there exists a $\delta>0$ so that
\begin{align}
F(t)
&\ge e^{-(C_{0}+\eps)/t}\quad\text{for all }t\in(0,\delta),\label{lb1}\\
F_{1}(t)
&\le e^{-(C_{k}-\eps)/t}\quad\text{for all }t\in(0,\delta).\label{lb2}
\end{align}
Using \eqref{lb1} and integrating by parts yields
\begin{align}\label{ibp055}
\begin{split}
I_{0,\delta}
&\le \int_{0}^{\delta}\big(1-e^{-(C_{0}+\eps)/t}\big)^{N-1}\,\dd F_{1}(t)\\
&= F_{1}(\delta)\big(1-e^{-(C_{0}+\eps)/\delta}\big)^{N-1}\\
&\quad
+ (N-1)\int_{0}^{\delta}(C_{0}+\eps)t^{-2}e^{-(C_{0}+\eps)/t}F_{1}(t)\big(1-e^{-(C_{0}+\eps)/t}\big)^{N-2}\,\dd t.
\end{split}
\end{align}
The first term in the righthand side of \eqref{ibp055} vanishes exponentially fast as $N\to\infty$. To handle the second term, we note that \eqref{lb2} implies that
\begin{align}
&\int_{0}^{\delta}t^{-2}e^{-(C_{0}+\eps)/t}F_{1}(t)\big(1-e^{-(C_{0}+\eps)/t}\big)^{N-2}\,\dd t\nonumber\\
&\quad\le\int_{0}^{\delta}t^{-2}e^{-(C_{0}+C_{k})/t}\big(1-e^{-(C_{0}+\eps)/t}\big)^{N-2}\,\dd t.\label{to9}
\end{align}
Applying Proposition~\ref{p1} to \eqref{to9} and using \eqref{ibp055} and the fact that $I_{\delta,\infty}$ vanishes exponentially fast as $N\to\infty$ completes the proof of \eqref{ub}.

To prove \eqref{both}, we use that \eqref{equal} ensures the existence of a $\delta'>0$ so that
\begin{align}
F(t)
&\le e^{-(C_{0}-\eps)/t}\quad\text{for all }t\in(0,\delta'),\label{lb3}\\
F_{1}(t)
&\ge e^{-(C_{k}+\eps)/t}\quad\text{for all }t\in(0,\delta').\label{lb4}
\end{align}
Using \eqref{lb3} and integrating by parts yields
\begin{align}\label{ibp1}
\begin{split}
I_{0,\delta'}
&\ge \int_{0}^{\delta'}\big(1-e^{-(C_{0}-\eps)/t}\big)^{N-1}\,\dd F_{1}(t)\\
&= F_{1}(\delta')\big(1-e^{-(C_{0}-\eps)/\delta'}\big)^{N-1}\\
&\quad
+ (N-1)\int_{0}^{\delta'}(C_{0}-\eps)t^{-2}e^{-(C_{0}-\eps)/t}F_{1}(t)\big(1-e^{-(C_{0}-\eps)/t}\big)^{N-2}\,\dd t.
\end{split}
\end{align}
The first term in the righthand side of \eqref{ibp1} vanishes exponentially fast as $N\to\infty$. To handle the second term, we note that \eqref{lb4} implies that
\begin{align}
&\int_{0}^{\delta}t^{-2}e^{-(C_{0}-\eps)/t}F_{1}(t)\big(1-e^{-(C_{0}-\eps)/t}\big)^{N-2}\,\dd t\nonumber\\
&\quad\ge\int_{0}^{\delta}t^{-2}e^{-(C_{0}+C_{k})/t}\big(1-e^{-(C_{0}-\eps)/t}\big)^{N-2}\,\dd t.\label{to10}
\end{align}
Applying Proposition~\ref{p1} to \eqref{to10} and using \eqref{ibp1} and the fact that $I_{\delta,\infty}$ vanishes exponentially fast as $N\to\infty$ completes the proof of \eqref{both}.
\end{proof}

\subsection{Numerical simulations}\label{numerical}

We now describe the numerical methods used to compute the extreme hitting probabilities in section~\ref{examples}.

\subsubsection{Pure diffusion in one dimension}

For the example of pure diffusion in $(0,l)$ considered in section~\ref{ex1d}, the probability density for hitting the right boundary is 
\begin{align}\label{f1}
f_{1}(t)
:=\frac{\dd}{\dd t}F_{1}(t)
=\frac{D}{l^{2}}\phi\Big(\frac{D}{l^{2}}t,1-\frac{x_{0}}{l}\Big),
\end{align}
where \cite{navarro2009}
\begin{align}\label{phi}
\phi(s,w)
:=
\begin{cases}
\sum_{k=1}^{\infty}\exp\big({-k^{2}\pi^{2}s}\big) 2k\pi\sin(k\pi w),\\
\frac{1}{\sqrt{4\pi s^{3}}}\sum_{k=-\infty}^{\infty}(w+2k)\exp\big(\frac{-(w+2k)^{2}}{4s}\big).
\end{cases}
\end{align}
The two representations for $\phi$ in \eqref{phi} are equivalent; the first is called the large-time expansion because it converges rapidly for large $s$ and the second is called the short-time expansion because it converges rapidly for small $s$. Integrating \eqref{f1} yields
\begin{align*}
F_{1}(t)
&=\int_{0}^{t}f_{1}(t')\,\dd t'
=\Phi\Big(\frac{D}{l^{2}}t,1-\frac{x_{0}}{l}\Big),
\end{align*}
where the large-time and short-time expansions of $\Phi$ are
\begin{align*}
\Phi(s,w)
=\int_{0}^{s}\phi(s',w)\,\dd s'
=
\begin{cases}\sum_{k=1}^{\infty}(1-e^{-k^{2}\pi^{2}s})\frac{2}{k\pi}\sin(k\pi w),\\
\sum_{k=-\infty}^{\infty}\textup{sgn}(2k+w)\textup{erfc}\Big(\frac{|2k+w|}{\sqrt{4s}}\Big),
\end{cases}
\end{align*}
and $\textup{erfc}(z):=1-\frac{2}{\sqrt{\pi}}\int_{0}^{z}e^{-u^{2}}\,\dd u$ denotes the complementary error function.

By symmetry, the probability density for hitting the left boundary is
\begin{align*}
f_{0}(t)
:=\frac{\dd}{\dd t}F_{0}(t)
=\frac{D}{l^{2}}\phi\Big(\frac{D}{l^{2}}t,\frac{x_{0}}{l}\Big),
\end{align*}
and
\begin{align*}
F_{0}(t)
&=\int_{0}^{t}f_{0}(t')\,\dd t'
=\Phi\Big(\frac{D}{l^{2}}t,\frac{x_{0}}{l}\Big).
\end{align*}
Hence,
\begin{align*}
F(t)
=F_{0}(t)
+F_{1}(t)
=\Phi\Big(\frac{D}{l^{2}}t,\frac{x_{0}}{l}\Big)
+\Phi\Big(\frac{D}{l^{2}}t,1-\frac{x_{0}}{l}\Big).
\end{align*}
Using these expressions, it is straightforward to derive the short-time behavior of $F$ and $F_{1}$ in \eqref{shortF}-\eqref{shortfarB}.

We use these formulas to numerically approximate the extreme hitting probabilities using the integral representation in Proposition~\ref{pint} and the trapezoidal rule. We use the short-time (large-time) expansions of $\phi$ and $\Phi$ for $s\le1$ ($s>1$) and using $10^{3}$ terms in these series representations. We take $D=l=1$.

\subsubsection{Diffusion with drift in one dimension}

For the example of diffusion with drift $\mu\in\R$ in $(0,l)$ considered in section~\ref{drift}, the probability density for hitting the left boundary is \cite{navarro2009}
\begin{align*}
f_{0}^{({\mu})}(t)
:=\frac{\dd}{\dd t}F_{0}(t)
=\exp\Big(-\frac{{\mu}x_{0}}{2D}-\frac{{\mu}^{2}t}{4D}\Big)f_{0}(t),
\end{align*}
and the density for hitting the right boundary is
\begin{align*}
f_{1}^{({\mu})}(t)
:=\frac{\dd}{\dd t}F_{1}(t)
=\exp\Big(\frac{{\mu}(l-x_{0})}{2D}-\frac{{\mu}^{2}t}{4D}\Big)f_{1}(t).
\end{align*}
Integrating these expressions yields
\begin{align*}
F_{0}(t)
&=\int_{0}^{t}f_{0}^{({\mu})}(t)\,\dd t
=\exp\Big(-\frac{{\mu}x_{0}}{2D}\Big)\Phi^{({\mu})}\Big(\frac{D}{l^{2}}t,\frac{x_{0}}{l}\Big),\\
F_{1}(t)
&=\int_{0}^{t}f_{1}^{({\mu})}(t)\,\dd t
=\exp\Big(\frac{{\mu}(l-x_{0})}{2D}\Big)\Phi^{({\mu})}\Big(\frac{D}{l^{2}}t,1-\frac{x_{0}}{l}\Big),
\end{align*}
where
\begin{align*}
\Phi^{({\mu})}(s,w)
=\sum_{k=1}^{\infty}\Big(1-\exp\big(-(b+k^{2}\pi^{2})s\big)\Big)\frac{2k\pi}{b+k^{2}\pi^{2}}\sin(k\pi w),
\end{align*}
and $b=\frac{l^{2}{\mu}^{2}}{4D^{2}}$. Hence,
\begin{align*}
F(t)
=F_{0}(t)
+F_{1}(t)
&=\exp\Big(-\frac{{\mu}x_{0}}{2D}\Big)\Phi^{({\mu})}\Big(\frac{D}{l^{2}}t,\frac{x_{0}}{l}\Big)\\
&\quad+\exp\Big(\frac{{\mu}(l-x_{0})}{2D}\Big)\Phi^{({\mu})}\Big(\frac{D}{l^{2}}t,1-\frac{x_{0}}{l}\Big).
\end{align*}
Using these expressions, it is straightforward to derive the short-time behavior of $F$ and $F_{1}$ in \eqref{shortF} and \eqref{shortfar} with the values in \eqref{driftA}-\eqref{driftB}.

We use these formulas to numerically approximate the extreme hitting probabilities using the integral representation in Proposition~\ref{pint} and the trapezoidal rule. We use the large-time expansions of $\phi$ and $\Phi^{(\mu)}$ for $s>1$ and the short-time expansion of $\phi$ for $s\le1$. For $s<1$, we numerically integrate (trapezoidal rule) $\Phi(\mu)(s,w)$ using the short-time expansion of $\phi$. We use $10^{3}$ terms in all these series representations. We take $D=l=1$.

\subsubsection{Partially absorbing targets}

For the example in section~\ref{partial} of partially absorbing targets, we numerically approximate the extreme hitting probabilities using the integral representation in Proposition~\ref{pint} and the trapezoidal rule. To obtain the values of $F(t)$ and $F_{1}(t)$ needed to compute these integrals, we numerically approximate the solution to the PDEs these distributions satisfy.

In particular, if we incorporate the starting position of the searcher into the definition of $F$,
\begin{align*}
F(x,t)
:=\P(\tau\le t\,|\,X(0)=x),
\end{align*}
then the initial-boundary value problem satisfied by $F$ is immediate from \eqref{ppde}-\eqref{pbcs} upon noting that $F=1-S$ where $S$ is defined in \eqref{pS}. Similarly, if we incorporate the starting position of the searcher into the definition of $F_{1}$,
\begin{align*}
F_{1}(x,t)
:=\P(\tau\le t\cap\kappa=1\,|\,X(0)=x),
\end{align*}
then $F_{1}$ satisfies that same PDE initial-boundary value problem as $F$ except that the initial-boundary conditions are
\begin{align*}
D\frac{\partial}{\partial x}F_{1}
&=\gamma_{0}F_{1},\quad x=0,\\
-D\frac{\partial}{\partial x}F_{1}
&=\gamma_{1}(1-F_{1}),\quad x=l.
\end{align*}
We approximate $F$ and $F_{1}$ by solving these PDE initial-boundary value problems using the Matlab PDE solver \texttt{pdepe} \cite{matlab}.

The conjectured short-time behavior in \eqref{shortF} and \eqref{shortfar} with parameters in \eqref{ac1}-\eqref{ac3} is derived in the following way. If we take $l\to\infty$, then it is straightforward to check that 
\begin{align*}
F(x,t)
=\text{erfc}(\frac{x}{\sqrt{4D t}})-e^{\gamma_{0}  (\gamma_{0}  t+x)/D} \text{erfc}\Big(\frac{2 \gamma_{0}  t+x}{\sqrt{4D t}}\Big),
\end{align*}
from which we can obtain \eqref{ac1}. The analogous argument yields \eqref{ac3}.

\subsubsection{Concentric targets in three dimensions}

For the example in section~\ref{3d} of diffusion between concentric spherical targets in three dimensions, we numerically approximate the extreme hitting probabilities using the integral representation in Proposition~\ref{pint} and the trapezoidal rule. To obtain the values of $F(t)$ and $F_{1}(t)$ needed to compute these integrals, we numerically approximate the solution to the PDEs these distributions satisfy.

In particular, if we incorporate the starting radial position of the searcher into the definition of $F$,
\begin{align*}
F(r,t)
:=\P(\tau\le t\,|\,\|X(0)\|=r),
\end{align*}
then it is well-known that $F(r,t)$ satisfies the diffusion equation,
\begin{align}\label{de}
\frac{\partial}{\partial t}F
=D\Big(\frac{2}{r}\frac{\partial}{\partial r}F+\frac{\partial^{2}}{\partial r^{2}}F\Big),\quad r\in(R_{0},R_{1}),
\end{align}
with zero initial condition, $F=0$ at $t=0$, and inhomogeneous boundary conditions, $F=1$ at $r\in\{R_{0},R_{1}\}$. If we similarly incorporate the starting radial position of the searcher into the definition of $F_{1}$,
\begin{align*}
F_{1}(r,t)
:=\P(\tau\le t\cap\kappa=1\,|\,\|X(0)\|=r),
\end{align*}
then $F_{1}$ also satisfies \eqref{de} with zero initial condition. The difference is that $F_{1}$ satisfies the boundary conditions $F_{1}=0$ at $r=R_{0}$ and $F_{1}=1$ are $r=R_{1}$. We approximate $F$ and $F_{1}$ by solving these PDE initial-boundary value problems using the Matlab PDE solver \texttt{pdepe} \cite{matlab}.

The conjectured short-time behavior in \eqref{sc1}-\eqref{sc4} is derived in the following way. If we take $R_{1}\to\infty$, then it is straightforward to check that 
\begin{align*}
F(r,t)
=\frac{R_{0}}{r}\,\textup{erfc}\Big(\frac{r-R_{0}}{\sqrt{4Dt}}\Big),
\end{align*}
from which we can obtain the short-time behavior in \eqref{sc1} and \eqref{sc3}. Similarly, if we take $R_{0}\to0$, then the short-time behavior of $F_{1}$ given in \eqref{sc2} and \eqref{sc4} is well-known and can be found in, for example, \cite{grebenkov2010subdiffusion}.

\bibliography{library.bib}
\bibliographystyle{unsrt}

\end{document}